\documentclass[12pt,reqno]{amsart}
\usepackage[margin=1in]{geometry}
\usepackage{amssymb,amsmath,enumerate,amsthm}
\usepackage[all]{xy}
\usepackage{mathtools}
\usepackage{hyperref}
\usepackage{graphicx}
\usepackage{amssymb}
\usepackage{epstopdf}
\usepackage[table]{xcolor}

\makeatletter
\@namedef{subjclassname@2020}{\textup{2020} Mathematics Subject Classification}
\makeatother

\usepackage[utf8]{inputenc}

\usepackage{dsfont}

\usepackage{indentfirst}
\usepackage{latexsym}

\usepackage{tikz}
\usetikzlibrary{intersections,backgrounds}
\usetikzlibrary{decorations.markings}

\newtheorem{Con}{Conjecture}
\newtheorem{thm}{Theorem}[section]
\newtheorem{lemma}[thm]{Lemma}
\newtheorem{prop}[thm]{Proposition}
\newtheorem{defn}[thm]{Definition}
\newtheorem{remark}[thm]{Remark}
\newtheorem{corollary}[thm]{Corollary}
\newtheorem{ex}[thm]{Example}

\newtheorem{prob}[thm]{Problem}

\usepackage{array} 
\newcommand{\PreserveBackslash}[1]{\let\temp=\\#1\let\\=\temp}
\newcolumntype{C}[1]{>{\PreserveBackslash\centering}p{#1}}
\newcolumntype{R}[1]{>{\PreserveBackslash\raggedleft}p{#1}}
\newcolumntype{L}[1]{>{\PreserveBackslash\raggedright}p{#1}}

\renewcommand{\a}{\alpha}
\newcommand{\Aut}{\textnormal{Aut}}
\renewcommand{\b}{\beta}
\newcommand{\B}{\textnormal{B}}

\newcommand{\ct}{\centering}

\newcommand{\Cl}{\textnormal{Cl}}

\newcommand{\D}{\Delta}

\newcommand{\End}{\textnormal{End}}

\newcommand{\E}{\mathcal{E}}

\newcommand{\F}{\mathbb{F}}

\newcommand{\Gal}{\textnormal{Gal}}
\newcommand{\Hom}{\textnormal{Hom}}
\renewcommand{\P}{\mathfrak{P}}
\newcommand{\GB}{G_{\B}}
\renewcommand{\k}{\Bbbk}
\newcommand{\N}{\mathbb{N}}
\newcommand{\Norm}{\textnormal{Norm}}
\renewcommand{\O}{\mathcal{O}}
\newcommand{\qq}{(q_1,q_2)}
\newcommand{\Q}{\mathbb{Q}}
\newcommand{\se}{\subseteq}

\newcommand{\Z}{\mathbb{Z}}
\usepackage{mathtools}

\newcommand{\Eleni}[1]{\textcolor{blue}{{\sf (Eleni:} {{#1})}}}
\newcommand{\Daniele}[1]{\textcolor{teal}{#1}}

\begin{document}

\title[Elliptic curves over Hasse pairs]{Elliptic curves over Hasse pairs}



\author{Eleni Agathocleous}
\address{\parbox{\linewidth}{Max-Planck-Institut f\"ur Mathematik, \\ Vivatsgasse 7, D-53111 Bonn, Germany.}}
\email{agathocleous@mpim-bonn.mpg.de}

\author{Antoine Joux}
\address{\parbox{\linewidth}{CISPA Helmholtz Center for Information Security, \\ Stuhlsatzenhaus 5, 66123 Saarbr\"ucken, Germany.}}
\email{joux@cispa.de}

\author{Daniele Taufer} 
\address{\parbox{\linewidth}{KU Leuven, \\ Celestijnenlaan 200A - box 2402, 3001 Leuven, Belgium.}}
\email{daniele.taufer@kuleuven.be}

\subjclass[2020]{11G07, 11G15, 14K02}

\date{\today}

 

\maketitle 

\vspace{-0.5cm}
\begin{center} with an appendix by \\ \MakeUppercase{\footnotesize{Pieter Moree and Efthymios Sofos}}\end{center}
\begin{abstract} We call a pair of distinct prime powers $\qq = (p_1^{a_1},p_2^{a_2})$ a Hasse pair if $|\sqrt{q_1}-\sqrt{q_2}| \leq 1$. For such pairs, we study the relation between the set $\E_1$ of isomorphism classes of elliptic curves defined over $\F_{q_1}$ with $q_2$ points, and the set $\E_2$ of isomorphism classes of elliptic curves over $\F_{q_2}$ with $q_1$ points. 
When both families $\E_i$ contain only ordinary elliptic curves, we prove that their isogeny graphs are isomorphic. When supersingular curves are involved, we describe which curves might belong to these sets. We also show that if both the $q_i$'s are odd and $\E_1 \cup \E_2 \neq \emptyset$, then $\E_1 \cup \E_2$ always contains an ordinary elliptic curve.
Conversely, if $q_1$ is even, then $\E_1 \cup \E_2$ may contain only supersingular curves precisely when $q_2$ is a given power of a Fermat or a Mersenne prime. In the case of odd Hasse pairs, we could not rule out the possibility of an empty union $\E_1 \cup \E_2$, but we give necessary conditions for such a case to exist. In an appendix, Moree and Sofos consider how frequently Hasse pairs occur using analytic number theory, making a connection with Andrica's conjecture on the difference between consecutive primes. \end{abstract}


\section{Introduction}
Given an elliptic curve $E$ over a finite field $\F_q$ of characteristic $p$, its number of points $|E(\F_q)|$ is directly related with a number $t$, called the trace of $E/\F_q$, which is defined as
\[ t = q + 1 - |E(\F_q)|. \]
All elliptic curves over $\F_q$ with the same number of rational points have the same trace, which we will denote by $t$.
We can then consider the set $\E(t)$ of $\F_q$-isomorphism classes of elliptic curves $E/\F_q$ with trace $t$:
\[ \E(t) = \{ E/\F_q \ : \ |E(\F_q)| = q + 1 -t\}/ \{ \F_{q}\textnormal{-isomorphisms} \}. \]

It is of both mathematical and cryptographic interest, to study such families $\E(t)$ over some finite field $\F_q$, when the number of points, and therefore the trace $t$, is given.
In this paper, we are interested in a `dual formulation' of this problem: given two prime powers $q_1, q_2$, we investigate the two sets
\[
    \E_1 = \{ E_1/\F_{q_1} \ : \ |E_1(\F_{q_1})| = q_2 \} / \{ \F_{q_1}\textnormal{-isomorphisms} \},
\]
and
\[
    \E_2 = \{ E_2/\F_{q_2} \ : \ |E_2(\F_{q_2})| = q_1 \} / \{ \F_{q_2}\textnormal{-isomorphisms} \}.
\]
To exclude trivial cases, we further ask that the two prime powers $q_1, q_2$ be different and satisfy the relation 
\[ |\sqrt{q_1}-\sqrt{q_2}| \leq 1. \]
When this relation holds true, we call such a pair $(q_1,q_2)$ a \emph{Hasse pair}.
We remark that the case $|\sqrt{q_1}-\sqrt{q_2}| > 1$ is degenerate: when $q_1,q_2$ do not belong to the Hasse interval of each other, then both sets $\E_i$ are necessarily empty.

The integer $t_1=q_1+1-q_2$ (resp. $t_2=q_2+1-q_1$) will be referred to as the trace of $\E_1$ (resp. $\E_2$), as it equals the trace of every curve (if any) in the corresponding set.

When $E_1 \in \E_1$ and $E_2 \in \E_2$ both arise as good reductions of a global curve $E$, the pair $\qq$ was called \emph{amicable} or $2$-\emph{aliquot cycle} for $E$ \cite{AmicableOverNF,SilvStan}, and the distribution of amicable primes for a given $E$ was notably shown to depend on whether $E$ has complex multiplication \cite{AmicableDistribution1,AmicableDistribution2,SilvStan}.
If the $E_i$'s share the same endomorphism ring $\O$, the pair $\qq$ has been referred to as \emph{dual elliptic} associated to $\O$ \cite{DualElliptic}.
In cryptography, such pairs of curves $(E_1,E_2)$ are called \emph{2-cycles}, and may be employed to design recursive zk-SNARK protocols \cite{SNARK}. Those are particularly efficient when the considered curves are pairing-friendly, which poses strong constraints on the feasible curves \cite{CyclesRevisited,CyclesLimitations}.
Examples of such applications appeared, among others, in \cite{TweedleCurves,PlutoErisCurves,PastaCurves,BN382curves}, and we refer the interested reader to \cite{SurveyCycles} for an extensive survey on this topic.

\subsection{Novel contributions}

Unlike previous related works, in this paper we do not start from a global curve. We observe instead that for a given Hasse pair, the associated traces $t_i$ are linearly dependent (Remark~\ref{rmk:sumtraces}) and they have the same discriminant $\D$ (Lemma~\ref{lem:sameD}), which prescribes whether the corresponding elliptic curves are ordinary or supersingular (Lemma~\ref{lem:TypeSqr}).
Furthermore, even though the set $\E_i$ may occasionally be empty, even when $\qq$ is a Hasse pair, this never occurs when the $q_i$'s are pure primes (Lemma \ref{lem:pq}); and in this case, both sets $\E_i$ consist of ordinary elliptic curves, except for a few small exceptions (Remark~\ref{rmk:pp}).

For every choice of odd Hasse pair $(p_1^{a_1}, p_2^{a_2})$, at least one of the $p_i$'s is proved to split completely in 
$\Q(\sqrt{\D})$ (Proposition~\ref{PrimeSplits}).

For any Hasse pair $\qq$, both the sets $\E_i$ are likely to consist of ordinary elliptic curves, and in this case, the directed multigraphs of their prime degree $\F_{q_i}$-isogenies are isomorphic (Theorem~\ref{thm:graphsIso}).
However, special choices of the $q_i$'s may lead to different configurations.
When both the $q_i$'s are odd, at most one of the sets $\E_i$ may contain supersingular elliptic curves (Proposition~\ref{prop:ss}), while this fails with even characteristic in a specific list of cases involving Fermat and Mersenne primes (Theorem~\ref{thm:even-ss}). 
Moreover, if $q_1$ is even, then the existence of a curve in $\E_1$ is sufficient to guarantee that $\E_2$ is also non-empty (Lemma~\ref{lem:ord-ord2} and Theorem~\ref{thm:even-ss}).

The only open case left is $\E_1 \cup \E_2 = \emptyset$.
In even characteristic this can happen, but we found only one such instance, namely $\qq = (2^8, 3^5)$.
We did not find any such a pair when both the $q_i$'s are odd, and we proved that such an example (if existent) should satisfy strict numerical conditions (Proposition~\ref{NotEmpty}).

\begin{prob} \label{prob:empty-empty}
    Characterizing Hasse pairs $\qq$ such that $\E_1$ and $\E_2$ are both empty.
    In particular, deciding if the set of such pairs is finite (or even empty, in odd characteristics).
\end{prob}
To the best of our knowledge, Problem \ref{prob:empty-empty} is open, except for cases of very small exponents ruled out by Corollary~\ref{cor:TwoPower}.
The possible configurations of sets $\E_i$ discussed in the present paper are displayed by Tables \ref{table:odd-odd} and \ref{table:even-odd}.

Finally, given the interesting properties that these sets of elliptic curves and their graphs exhibit, it is natural to ask, not only out of mathematical curiosity but also from a cryptographic perspective, whether there are many Hasse pairs.
This question is directly related to a long-standing conjecture known as Andrica's conjecture, which states that the difference of the square roots of any two consecutive primes is always less than one. In Appendix~\ref{Ap:PietSof}, Moree and Sofos discuss this conjecture in the context of prime powers, and they prove an asymptotic result regarding the number of Hasse pairs in which a given prime $p$ can be a member.

\smallskip

\subsection{Notation}
We denote the set of natural numbers by $\N = \{0,1,\dots\}$, and its positive subset by $\Z_{>0} = \{1,2,\dots\}$.
We write $p^a || N$ when $p^a|N$ but $p^{a+1} \nmid N$.
When the order of the elements of a pair is irrelevant, we denote it by $(\qq)$ to indicate $\qq$ or $(q_2,q_1)$.

\subsection{Paper organization}

In Section~\ref{sec:Intro2} we recall a theorem of Waterhouse on valid traces, and we prove a corollary of this result regarding elliptic curves of $j$-invariant zero. In Section~\ref{HassePairs}, we define Hasse pairs and we establish several properties of its associated discriminant $\D$ and traces $t_1$ and $t_2$.
In Section~\ref{non-ordinary cases} we focus on the non-ordinary cases, i.e. when at least one of the sets $\E_1$ or $\E_2$ is empty or is comprised of supersingular elliptic curves.
Section~\ref{split} is devoted to the splitting behavior of the base primes of a given Hasse pair in the corresponding imaginary quadratic field $\Q(\sqrt{\D})$.
In Section~\ref{sec:ord-ord}, we deal with the ordinary case and we establish the correspondence between the isogeny graphs of the elliptic curves in $\E_1$ and of those in $\E_2$. 
Some numerical and visual examples of such graphs are presented in Section~\ref{examples}.
Finally, in Section~\ref{conclusions}, we summarise the results of the current paper, and we draw our conclusions.

\section{Valid Traces}\label{sec:Intro2}

Let $q \in \Z_{>0}$ be a prime power.
We consider all elliptic curves over $\F_q$ with the same number of points or, equivalently, with the same trace $t \in \Z$.
We denote the set of $\F_q$-isomorphism classes of such curves by $\E(t)$.
The following theorem of Waterhouse determines when $\E(t)$ is not empty and the type of elliptic curves it contains.

\begin{thm}[{\cite[Theorem 4.1]{Waterhouse}}] \label{thm:Waterhouse}
Let $\E(t)$ be the set of elliptic curves of trace $t$ over $\F_{q} = \F_{p^a}$. Then $\E(t)$ is non-empty precisely in the following cases.
\begin{enumerate}[$(a)$]
    \item \label{W:ordinary} $|t| \leq 2\sqrt{q}$, with $\gcd(t,q) = 1$.
    \item \label{W:sseven} $t = \pm 2\sqrt{q}$, and $a$ is even.
    \item \label{W:odd1} $t = \pm \sqrt{q}$, with $p \not\equiv 1 \bmod 3$ and $a$ is even.
    \item \label{W:odd2} $t = \pm \sqrt{p^{a+1}}$, with $p \in \{2,3\}$ and $a$ is odd.
    \item \label{W:0} $t = 0$, and $(\textnormal{I})$ $a$ is odd or $(\textnormal{II})$ $a$ is even, and $p \not\equiv 1 \bmod 4$.
\end{enumerate}
The elliptic curves of case~\eqref{W:ordinary} are ordinary.
The curves of case~\eqref{W:sseven} are supersingular and have all their endomorphisms defined over ${\F_{q}}$, while the rest are supersingular but do not have all their endomorphisms defined over ${\F_{q}}$.
\end{thm}

By \cite[Theorems 4.1 and 4.2]{Waterhouse}, one can deduce the values of the $j$-invariants for some of the cases of Theorem~\ref{thm:Waterhouse}.
Since we have not found this result in the literature, we state and prove it here, for completeness. 

\begin{prop}\label{prop:zero1728}
    The elliptic curves corresponding to cases \eqref{W:odd1} and \eqref{W:odd2} in Theorem~\ref{thm:Waterhouse} have $j$-invariant $0$.
\end{prop}
\begin{proof}
Since the CM-discriminant is $\D = t^2-4q$, we have
    \[t = \begin{cases}
    \pm p^{\frac{a}{2}} & \textnormal{ with } a \textnormal{ even,} \\
    \pm p^{\frac{a+1}{2}} & \textnormal{ with } p \in \{2,3\}, a \textnormal{ odd.} \\
\end{cases} \implies \D = \begin{cases}
    -3 q & \textnormal{ in case } \eqref{W:odd1}, \\
    (p-4)q & \textnormal{ in case } \eqref{W:odd2}. \\
\end{cases}
    \]
The elliptic curves corresponding to the above cases, do not have all their endomorphisms defined over $\F_q$, hence by \cite[Theorem 4.2-(3)]{Waterhouse} the possible endomorphism rings arising from the elliptic curves in $\E(t)$ have conductor coprime to $q$.

In case \eqref{W:odd1} the possible conductors are those dividing $\sqrt{q}$, therefore the unique possible order is the maximal one, i.e. the ring of integers of $\Q(\sqrt{-3})$, which is the endomorphism ring of curves with $j$-invariant $0$.

Similarly, in case \eqref{W:odd2}, if $p=2$ the possible conductors are those dividing $2^{(a+1)/2}$, therefore the unique possible order is the ring of integers of $\Q(i)$, which is defined by curves with $j$-invariant $1728 = 0 \in \F_q$. If $p=3$ the conductors need to divide $3^{(a-1)/2}$, therefore the unique endomorphism order is the ring of integers of $\Q(\sqrt{-3})$, which is again defined by curves with $j$-invariant $0$.
\endproof
\end{proof}


\section{Hasse pairs}\label{HassePairs}

Throughout the paper, $q_1=p_1^{a_1}$ and $q_2=p_2^{a_2}$ will denote arbitrary prime powers. Given such a pair $(q_1,q_2)$, we consider the (possibly empty) set of elliptic curves defined over $\F_{q_1}$ with $q_2$ rational points, and respectively the set of elliptic curves over $\F_{q_2}$ with $q_1$ points.
To avoid trivial repetitions, we consider these curves up to $\F_{q_i}$-isomorphism.

\begin{defn}\label{HasseDefn}
    Let $\qq$ be a pair of distinct prime powers.
    For $(i,j) = ((1,2))$, we define
    \[
    \E_i = \{ E_i/\F_{q_i} \ : \ |E_i(\F_{q_i})| = q_j \} / \{ \F_{q_i}\textnormal{-isomorphisms} \}.
    \]
    We also define the \emph{trace} $t_i$ corresponding to such a pair as
    \[ t_i = q_i + 1 - q_j. \]
\end{defn}

\begin{remark} \label{rmk:sumtraces}
The integers $t_i$ are the traces of the Frobenius morphism for every elliptic curve (if any) in $\E_i$ \cite[Remark V.2.6]{Silverman}.
Moreover, by adding their defining equations, we immediately see that
\begin{equation*}
    t_1 + t_2 = 2.
\end{equation*} 
\end{remark}

\begin{remark}
    The results of the present paper would hold even when $q_1 = q_2$. We have excluded this case from Definition~\ref{HasseDefn}, since it leads to the trivial relations $\E_1 = \E_2$ and $t_1=t_2=1$.
    These curves of trace $1$ are often called \emph{anomalous} curves, and are all ordinary by Theorem~\ref{thm:Waterhouse} (a).
\end{remark}

\begin{defn} \label{discriminant}
We define the \emph{discriminant} associated with the pair $\qq$ as
\[
    \D = (q_1-q_2)^2-2(q_1+q_2)+1.
\] 
\end{defn}
As we prove below, this is the maximal CM-discriminant of the curves in the sets $\E_i$, when they are not empty. 
An analogous result for amicable curves was already observed in \cite[Proposition 3.4]{SilvStan}.

\begin{lemma} \label{lem:sameD}
Let $t_1,t_2$ and $\D$ be defined as above. Then
\[ t_1^2 - 4q_1 = \D = t_2^2 - 4q_2. \]
\end{lemma}\label{EqualDiscs}
\proof By the definition of $t_1$ we have
\[ t_1^2 - 4q_1 = (q_1-q_2)^2 + 2(q_1-q_2) + 1 - 4q_1 = \D, \]
which is symmetric in $q_1$ and $q_2$, thus it also equals $t_2^2 - 4q_2$.
\endproof

Hasse's Theorem \cite[Theorem V.1.1]{Silverman} implies that the sets $\E_i$ can be non-empty if the considered $q_i$'s are close enough. This motivates the following definition.

\begin{defn}[Hasse pair] \label{defn:HassePair}
Let $q_1, q_2$ be distinct prime powers. We call the pair $\qq$ a \emph{Hasse pair} iff
\[ | \sqrt{q_1} - \sqrt{q_2} | \leq 1. \]
When both $q_1$ and $q_2$ are odd, we will refer to it as an \emph{odd Hasse pair}.
\end{defn}

The following lemma links the above definition to Hasse's theorem. 

\begin{lemma} \label{lem:Hasse}
Let $t_1, t_2$ be defined as above. The following are equivalent:
\begin{enumerate}
    \item\label{Heq:i} $(q_1,q_2)$ is a Hasse pair,
    \item\label{Heq:ii} $|t_1| \leq 2\sqrt{q_1}$,
    \item\label{Heq:iii} $|t_2| \leq 2\sqrt{q_2}$.
\end{enumerate}
\end{lemma}
\proof Condition \eqref{Heq:i} is equivalent to
\[ \sqrt{q_2} - 1 \leq \sqrt{q_1} \leq \sqrt{q_2} + 1. \]
Since all the quantities are positive, this is equivalent to its squared version, namely
\[ q_2 - 2 \sqrt{q_2} + 1 \leq q_1 \leq q_2 + 2 \sqrt{q_2} + 1. \]
This can be written as
\[ -2 \sqrt{q_2} \leq -(q_2 - q_1 + 1) \leq 2 \sqrt{q_2}, \]
which is condition \eqref{Heq:iii}.
Since condition \eqref{Heq:i} is symmetric, the equivalence between \eqref{Heq:i} and \eqref{Heq:ii} follows in the same way, by reversing the roles of $q_1$ and $q_2$.
\endproof

An elliptic curve over $\F_{q_i}$ is called \emph{supersingular} if its trace is divisible by $p_i$, and \emph{ordinary} otherwise. Since all the curves in $\E_i$ (if any) have the same trace $t_i$, then the set $\E_i$ is either empty or contains curves that have the same type (i.e. they are all ordinary or all supersingular). This motivates the following definition.

\begin{defn}
    We call the set $\E_i$ \emph{ordinary} (resp. \emph{supersingular}) if it is non-empty and contains ordinary (resp. supersingular) elliptic curves. 
\end{defn}

\begin{remark} \label{rmk:uniquej}
    We recall that two ordinary curves are isomorphic over $\F_q$ if and only if they have the same size and the same $j$-invariant \cite[Proposition 14.19]{Cox}. Hence, when the set $\E_i$ is ordinary, it can be identified with the set containing the $j$-invariants of all elliptic curves that belong to $\E_i$.
\end{remark}

\begin{lemma} \label{lem:TypeSqr}
Let $\qq$ be a Hasse pair. For both $i \in \{1,2\}$, the following hold.
\begin{itemize}
    \item If $\gcd(\D,q_i) = 1$, then $\E_i$ is ordinary.
    \item If $\gcd(\D,q_i) \neq 1$, then $\E_i$ is either supersingular or empty.
\end{itemize}
\end{lemma}
\proof We notice that
\[ \gcd(t_i,q_i) = 1 \iff \gcd(t_i^2-4q_i,q_i) = 1 \iff \gcd(\D,q_i) = 1. \]
All the curves of $\E_i$ (if any) have trace $t_i$ (Remark \ref{rmk:sumtraces}), hence they are ordinary when $\gcd(\D,q_i) = 1$, and supersingular when $\gcd(\D,q_i) \neq 1$.
Since $|t_i| \leq 2 \sqrt{q_i}$ by Lemma~\ref{lem:Hasse}, in the ordinary case the set $\E_i$ is always non-empty (see Theorem~\ref{thm:Waterhouse}, case \eqref{W:ordinary}).
\endproof

When dealing with pure primes, the sets $\E_i$ are almost always comprised of ordinary curves, as shown by the next lemma.

\begin{lemma} \label{lem:pq}
If $(p_1,q_2)$ is a Hasse pair, with $p_1$ prime, then $\E_1$ is non-empty.
Moreover, if $\E_1$ is supersingular then $q_2=p_1+1$ or $p_1\in\{2,3\}$.
\end{lemma}
\proof Since $\qq$ is Hasse, by Lemma~\ref{lem:Hasse} we have $|t_1| \leq 2\sqrt{p_1} < 2p_1.$
If $p_1 \in \{2,3\}$, the integer values lying between $-2p_1$ and $2p_1$ are either coprime to $p_1$ \eqref{W:ordinary}, or equal to $\pm p_1$ \eqref{W:odd2}, or equal to $0$ \eqref{W:0}.

When $p_1 > 3$,  we have the tighter bound $|t_1| \leq 2\sqrt{p_1} < p_1$.
In this case, the integers between $-p_1$ and $p_1$ are either coprime to $p_1$ \eqref{W:ordinary} or equal to $0$ \eqref{W:0}.
In all such cases, $\E_1$ is non-empty by Theorem~\ref{thm:Waterhouse}.
Furthermore, if $p_1 > 3$ the only supersingular case has $ t_1 = 0$ 
which implies $q_2 = p_1+1$.
\endproof

\begin{remark} \label{rmk:pp}
As a consequence of Lemma~\ref{lem:pq}, whenever $(p_1,p_2)$ is a Hasse pair consisting of pure primes (i.e. $a_1 = a_2 = 1$), in all but a few small cases, both sets $\E_i$ are ordinary.
A straightforward verification of those cases shows that $\E_1$ is actually supersingular precisely when $(p_1,p_2) \in \{ (2,3),(2,5),(3,7) \}$.
\end{remark}

\begin{ex} \label{ex:ord-ss}
Since $\sqrt{7}-\sqrt{3} = 0.913\dots$, then $(p_1,p_2) = (3,7)$ is an odd Hasse pair.
Up to isomorphism, there is only one elliptic curve in $\E_1$, defined by
$ y^2 = x^3 + 2x + 1. $
As $t_1 = 3 + 1 - 7 = -3$, this curve is supersingular. Similarly, there is only one curve in $\E_2$, defined by
$y^2 = x^3 + 4. $
It has trace $t_2 = 2 - t_1 = 5$, hence it is ordinary. Finally, we notice that the $j$-invariant of both curves is $0$, and this is not a coincidence as we show in Proposition~\ref{prop:j=0}.
\end{ex}


\section{Non-ordinary cases}\label{non-ordinary cases}

In this section, we discuss the relation between $\E_1$ and $\E_2$ when $\gcd(\D, q_1q_2) \neq 1$, namely when at least one of these sets is either supersingular or empty.

We begin by noticing that non-ordinary cases may occur only in different characteristics.

\begin{lemma} \label{lem:diffchar}
    If $p_1 = p_2$, then $\gcd(\D, q_1q_2) = 1$.
\end{lemma}
\proof If $p_1=p_2$, then $q_i = p_1^{e_i}$. 
Therefore
\[
    \D = (q_1-q_2)^2-2(q_1+q_2)+1 \equiv 1 \bmod p_1,
\]
which implies $\gcd(\D, q_1q_2) = 1$.
\endproof

\subsection{Odd characteristics}
When $q_1$ and $q_2$ are both odd, then so are $t_1$ and $t_2$.

\begin{prop} \label{prop:ss}
Let $\qq$ be an odd Hasse pair. If $\E_1$ is supersingular, then $\E_2$ is ordinary.
Thus, if $\E_1 \cup \E_2$ is non-empty, then it always contains an ordinary elliptic curve.
\end{prop}
\proof Let $E_1 \in \E_1$ be a supersingular curve. We first consider the case $p_1 \neq 3$. 
Since $t_1$ is odd, only case~\eqref{W:odd1} of Theorem~\ref{thm:Waterhouse} can apply. Thus, $a_1$ needs to be even and $t_1 = \pm \sqrt{q_1}$, hence
$\D = t_1^2 - 4q_1 = -3q_1.$
Since $\gcd(q_1,q_2) = 1$ by Lemma~\ref{lem:diffchar}, then $\gcd(\D,q_2) \in \{1,3\}$.

Let us assume by contradiction that $p_2=3$, then by Remark~\ref{rmk:sumtraces} we get
\[ 2 - t_1 = t_2 = 3^{a_2} + 1 - p_1^{a_1}, \quad
\mbox{which implies} \quad
p_1^{a_1} \mp p_1^{\frac{a_1}{2}} + 1 = 3^{a_2}. \]
If $a_2 = 1$, then $q_2 = 3$ and the only odd prime powers inside its Hasse interval are $5$ and $7$, both with odd exponents, contradicting the parity of $a_1$.
If $a_2 > 1$, then $x = -p_1^{\frac{a_1}{2}}$ should be a root of one of the polynomials
\[ x^2 \pm x + 1 \in \Z/9\Z[x],\]
which is impossible since those polynomials are both irreducible. 
Hence, we conclude that $p_2 \neq 3$, which implies $\gcd(\D,q_2) = 1$, then $\E_2$ is ordinary by Lemma~\ref{lem:TypeSqr}.

Let us now consider the case $p_1 = 3$. 
Since $t_1$ is odd, by Theorem~\ref{thm:Waterhouse} we have
\[t_1 = \begin{cases}
    \pm 3^{\frac{a_1}{2}} & \textnormal{if } a_1 \textnormal{ is even, \quad case~\eqref{W:odd1}}, \\
    \pm 3^{\frac{a_1+1}{2}} & \textnormal{if } a_1 \textnormal{ is odd, \ \quad case~\eqref{W:odd2}},
\end{cases}\]
which leads us to
\[ \D = \begin{cases}
    (3^{\frac{a_1}{2}})^2-4\cdot 3^{a_1} = -3^{a_1+1} & \textnormal{if } a_1 \textnormal{ is even,} \\
    (3^{\frac{a_1+1}{2}})^2-4\cdot 3^{a_1} = -3^{a_1} & \textnormal{if } a_1 \textnormal{ is odd.}
\end{cases}\]
By Lemma~\ref{lem:diffchar} we have $\gcd(q_1,q_2) = 1$, therefore $3 \nmid q_2$. Hence, $\D$ and $q_2$ cannot have common factors and we conclude that $\E_2$ is ordinary by Lemma~\ref{lem:TypeSqr}.

Thus, for odd Hasse pairs the sets $\E_i$ cannot be both supersingular, then if $\E_1 \cup \E_2$ is not empty, then it contains an ordinary elliptic curve.
\endproof

Next, we show that in odd characteristics, establishing the existence of supersingular curves in $\E_1 \cup \E_2$ only requires testing curves with $j$-invariant $0$.

\begin{prop} \label{prop:j=0}
Let $\qq$ be an odd Hasse pair. If $E_1 \in \E_1$ is supersingular, then $j(E_1)= 0$. Furthermore, in this case $\E_2$ is ordinary and it also contains an elliptic curve with $j$-invariant $0$.
\end{prop}
\proof Since $t_1$ is odd, we are in cases \eqref{W:odd1} or \eqref{W:odd2} of Theorem~\ref{thm:Waterhouse}, thus by Proposition~\ref{prop:zero1728} we have $j(E_1) = 0$.
Moreover, by Proposition~\ref{prop:ss}, $\E_2$ is ordinary and, by \cite[Theorem 4.2 (2)]{Waterhouse}, the maximal order always arises as the endomorphism ring of some curve in $\E_2$. We have
\[t_1 = \begin{cases}
    \pm p_1^{\frac{a_1}{2}} & \textnormal{if } a_1 \textnormal{ is even,} \\
    \pm 3^{\frac{a_1+1}{2}} & \textnormal{if } a_1 \textnormal{ is odd,}
\end{cases} \implies \D = \begin{cases}
    -3 p_1^{a_1} & \textnormal{if } a_1 \textnormal{ is even,} \\
    -3^{a_1} & \textnormal{if } a_1 \textnormal{ is odd.}
\end{cases}\]
We observe that the maximal order in both cases is the ring of integers of $\Q(\sqrt{-3})$, which is always isomorphic to the endomorphism ring of elliptic curves of $j$-invariant $0$.
\endproof

\subsection{Even characteristics}
In this section, we discuss Hasse pairs involving even prime powers, so we will assume $p_1 = 2$. 

\begin{lemma} \label{lem:ord-ord2}
Let $(2^{a_1},q_2)$ be a Hasse pair. The following are equivalent:
\begin{enumerate}
    \item\label{lem:ord-ord2-i} $\E_1$ is ordinary,
    \item\label{lem:ord-ord2-ii} $q_2$ is even,
    \item\label{lem:ord-ord2-iii} $(2^{a_1},q_2) \in \{ ((2,4)), ((4,8)) \}$.
\end{enumerate}
\end{lemma}
\proof \eqref{lem:ord-ord2-i} $\Rightarrow$ \eqref{lem:ord-ord2-ii}: Since $\E_1$ is ordinary, then $t_1$ is odd, so is $t_2$ (Remark \ref {rmk:sumtraces}). Therefore $q_2 = t_2-1+2^{a_1}$ is even.

$\eqref{lem:ord-ord2-ii} \Rightarrow \eqref{lem:ord-ord2-iii}$: A straightforward check shows that there are only two powers of $2$ lying in the Hasse bound of each other, namely $\{2,4\}$ and $\{4,8\}$.

$\eqref{lem:ord-ord2-iii} \Rightarrow \eqref{lem:ord-ord2-i}$: Every $(2^{a_1},q_2) \in \{ ((2,4)), ((4,8)) \}$ has discriminant $\D = -7$, which is coprime to $2^{a_1}$. Thus, $\E_1$ is ordinary by Lemma \ref{lem:TypeSqr}.
\endproof

The ordinary cases of Lemma \ref{lem:ord-ord2} are detailed in Examples \ref{ex:2-4} and \ref{ex:4-8}.
Except for such cases, if $\gcd(\D,q_1q_2) \neq 1$ we can assume that $q_2$ is odd and $\E_1$ is supersingular or empty.
In the supersingular case, we know from Deuring \cite[\S 8.2]{Deuring} that the only possible curves in $\E_1$ have $j$-invariant $0$, similarly to the odd case (Proposition \ref{prop:j=0}).


However, with an even characteristic one may simultaneously find supersingular curves in both $\E_1$ and $\E_2$.
To determine when this is the case we need the following lemma. It is an easy instance of the theorem of Mih\u{a}ilescu \cite{CatalanConj}, which proved the famous Catalan's conjecture, but the version we need can be proved with elementary techniques.

\begin{lemma}[{\cite[pp. 201-203]{EleniRef}}] \label{lem:consecutiveprimepowers}
Let $\a,\b \geq 2$ be integers and let $p \in \N$ be a prime.
If $2^\a$ and $p^\b$ are consecutive integers, then $\a = 3, \b = 2$ and $p = 3$.
\end{lemma}

\begin{thm} \label{thm:even-ss}
 Let $\qq = (2^{a_1},p_2^{a_2})$ be a Hasse pair.
If $\E_1$ is supersingular of trace $t_1$, then precisely one of the following holds.
\begin{enumerate}
    \item\label{ss-ss-i} $t_1 = \pm 2^{\frac{a_1+2}{2}}$ and 
    \[ \qq = \begin{cases}
        (2^6,3^4), &\textnormal{ or } \\
        \big( 2^{2^{n+1}},(2^{2^{n}}+1)^2 \big), &\textnormal{ with } p_2 = 2^{2^{n}}+1 \textnormal{ a Fermat prime, or } \\
        \big( 2^{2m},(2^m-1)^2 \big), &\textnormal{ with } p_2 = 2^m-1 \textnormal{ a Mersenne prime. }
    \end{cases} \]
    \item\label{ss-ss-ii} $t_1 = 2$ and $\qq = (2^2,3)$.
    \item\label{ss-ss-iii} $t_1 = 0$, or $a_1$ is odd and $t_1 = \pm 2^{\frac{a_1+1}{2}}$, or $a_1$ is even, $t_1 = \pm 2^{\frac{a_1}{2}}$ and $p_2 \neq 3$.
\end{enumerate}
    In cases \eqref{ss-ss-i} and \eqref{ss-ss-ii}, $\E_2$ is supersingular, while in case \eqref{ss-ss-iii}, $\E_2$ is ordinary.
\end{thm}

\proof Since $\E_1$ is supersingular, then $p_2$ is odd by Lemma \ref{lem:ord-ord2} and $t_1$ is given by one of the cases \eqref{W:sseven}-\eqref{W:0} of Theorem~\ref{thm:Waterhouse}.
We examine such cases separately.

Case \eqref{W:sseven}: $a_1$ is even and $t_1 = \pm 2^{\frac{a_1+2}{2}}$, thus $\D = t_1^2-4q_1 = 2^{a_1+2}-4 \cdot 2^{a_1} = 0$.
Since $\D = t_2^2-4p_2^{a_2}$ by Lemma~\ref{lem:sameD}, then $a_2$ is also even.
As $\qq$ is Hasse, this implies that $2^{\frac{a_1}{2}}$ and $p_2^{\frac{a_2}{2}}$ are consecutive integers.
By Lemma~\ref{lem:consecutiveprimepowers} if $\frac{a_1}{2},\frac{a_2}{2} \geq 2$ there is a unique possibility, namely $\qq = (2^6,3^4)$, corresponding to the consecutive integers $8$ and $9$.
Now consider the case where $a_1$ or $a_2$ is $2$.
If $a_1 = 2$, we straightforwardly verify that there are no Hasse pairs $(4,p_2^{a_2})$ with $a_2$ even.
If $a_2 = 2$, then $2^{\frac{a_1}{2}}$ and $p_2$ are consecutive integers, therefore $p_2$ is either a Fermat or a Mersenne prime.
It is easy to check that the corresponding traces $t_2$ fit into case \eqref{W:sseven} of Theorem~\ref{thm:Waterhouse}, hence in such cases $\E_2$ is supersingular.

Case \eqref{W:odd1}: $a_1$ is even and $t_1 = \pm 2^{\frac{a_1}{2}}$, then $\D = -3 \cdot 2^{a_1}$.
If $p_2 > 3$ then $\E_2$ is ordinary by Lemma~\ref{lem:TypeSqr}.
Let us consider $p_2 = 3$.
By Lemma~\ref{lem:sameD} we have
\[ -3 \cdot 2^{a_1} = t_2^2 - 4\cdot 3^{a_2}. \]
Since $-3$ is not a quadratic residue modulo $9$, the above equation may have integer solutions only if $a_2 = 1$, hence $q_2 = 3$. We straightforwardly check that the unique Hasse pair $(2^{a_1},3)$ with $a_1$ even is $(2^2,3)$, which gives $t_1 = 2, t_2 = 0$ and therefore $\E_2$ supersingular.

Cases \eqref{W:odd2} and \eqref{W:0}: If $a_1$ is odd and $t_1 = \pm 2^{\frac{a_1+1}{2}}$, or if $t_1 = 0$, then $\D$ is a power of $2$, hence $\E_2$ is ordinary by Lemma~\ref{lem:TypeSqr}.
\endproof

\begin{remark} \label{rmk:FermatMersenne}
    By Proposition \ref{prop:j=0} and Theorem \ref{thm:even-ss}, both sets $\E_i$ can be simultaneously supersingular only if one of the considered prime powers is even, and the other one involves primes of Fermat or Mersenne.
    On the other side, if the Lenstra-Pomerance-Wagstaff conjecture \cite{LPW} holds true, this would imply the existence of infinitely many Mersenne primes, hence by Theorem \ref{thm:even-ss} both sets $\E_i$ are supersingular infinitely many times.
\end{remark}

\begin{remark} When both sets $\E_i$ are supersingular, they are often unrelated.
For instance, for $\qq = (2^{10}, 31^2)$ the set $\E_1$ contains one curve of $j$-invariant $0 \in \F_{2^{10}}$, while $\E_2$ contains three curves of $j$-invariants $2,4,23 \in \F_{31^{2}}$.
Instead, in the next sections we will prove that this is never the case when both sets $\E_i$ are ordinary.
\end{remark}

\section{The splitting behavior of the primes in odd Hasse pairs} \label{split}

In this section, we consider odd Hasse pairs $\qq$, and therefore the traces $t_i$ are both odd integers. 

\begin{lemma} \label{PropertiesD}
Let $\qq$ be an odd Hasse pair.
Then $\D < 0$, $\D \equiv 1 \bmod 4$ and $|\D|$ is not a perfect square.
\end{lemma}
\proof Since $\qq$ is Hasse and $t_1$ is odd, by Lemma~\ref{lem:Hasse} we have $t_1 < 2 \sqrt{q_1}$, hence by Lemma~\ref{lem:sameD} we conclude
\[ \D = t_1^2 - 4 q_1 < 0. \]
Since $t_1^2 \equiv 1 \bmod 4$, the above equation implies $\D \equiv 1 \bmod 4$, so $|\D| = -\D \equiv -1 \bmod 4$.
Therefore $|\D|$ cannot be a square, as it is not a quadratic residue modulo $4$.
\endproof

We can write the discriminant $\D$ associated with $\qq$ as 
\[ \D = f^{2}D, \]
where $f$ is called the \emph{conductor}, and $D$ is the squarefree part of $\D$.
By Lemma~\ref{PropertiesD}, we have $D \equiv 1\bmod 4$ as well, hence $D$ is a fundamental discriminant.
We denote by $K$ the imaginary quadratic field 
\[ K = \Q\big(\sqrt{\D}\big) = \Q\big(\sqrt{D}\big), \] 
and by $\O_{K}$ the maximal order of $K$.

\begin{remark} \label{rmk:no1728}
    When $\qq$ is an odd Hasse pair, we observe that we never have $K = \Q(i)$, since $\D \equiv 1 \bmod 4$. Therefore, in this case, the $j$-invariant $1728$ cannot correspond to an elliptic curve in either set $\E_i$.
\end{remark}

\begin{remark} \label{Norm}
By Lemma \ref{lem:sameD} we have $\D = t_i^{2} - 4q_i$, from which we have
\[ \label{RelnNorm} q_i = \frac{{t_i}^{2} - \D}{4} = \Norm_{K/\mathbb{Q}}\left(\frac{t_i + \sqrt{\D}}{2}\right). \]
It shows that $q_i$ is the norm from $K$ down to $\Q$, of the integral element 
\[ \pi_{i} = \frac{t_i + \sqrt{\D}}{2} \in \O_{K}. \]
Furthermore, by Remark \ref{rmk:sumtraces}, we obtain the relation \begin{equation} \label{eq:homothetic}
 \pi_i = 1 - \overline{\pi_j}, \qquad i \neq j. \end{equation}
\end{remark}

\begin{remark} \label{rmk:FrobsPrim}
Since $\D$ is not a square, the polynomial $f_{i}(x) = x^{2} - t_ix + q_i \in \Z[x]$ is irreducible for $i \in \{1,2\}$, and it is the minimal polynomial of the element $\pi_{i} \in \O_K$. 
In particular, $\mathbb{Z}[\pi_{i}]$ are both suborders of $\O_{K}$.
\end{remark}

\begin{remark}\label{rmk:SameOrders}
    The discriminant $\delta_{i}$ of each order $\mathbb{Z}[\pi_{i}]$ can be easily computed as
    \[ \delta_{i} = \left(\det\begin{pmatrix}
1 & \pi_{i} \\
1 & \overline{\pi}_{i} 
\end{pmatrix}\right)^{2} = \D = f^{2}D. \]
Since 
$\pi_1-\pi_2 = 1 - t_2$ by Remark \ref{rmk:sumtraces}, we also have $\mathbb{Z}[\pi_{1}] = \O_{\D} = \mathbb{Z}[\pi_{2}]$.
\end{remark}

\begin{prop} \label{PrimeSplits}
Let $\qq$ be an odd Hasse pair.
Then at least one of the base primes $p_i$ splits completely in $K$.
\end{prop}
\proof We equivalently show that in $\O_{K}$: (a) the $p_i$'s cannot be both inert, and (b) if $p_i$ ramifies, then $p_j$ splits for $j \neq i$.

Case (a): The fact that $\Norm_{K/\mathbb{Q}}(\pi_{i}) = q_i = p_i^{a_{i}}$ implies that $p_i$ is the unique prime of $\O_K$ that divides $\pi_{i}$, since $\O_{K}$ is Dedekind and we have unique factorization of ideals.
Thus, for both $i\in\{1, 2\}$ there are positive integer $b_i \in \Z_{>0}$ such that
\begin{equation*}
    (\pi_{i}) = (p_i^{b_{i}}).
\end{equation*}
If both primes $p_i$ were inert in $K$, by taking norms in the above equation, we would have $p_i^{a_i} = q_i = \Norm_{K/\mathbb{Q}}(\pi_{i}) = p_i^{2b_{i}}$.
This in turn gives us that $a_i = 2b_i$ should be both even.
Since $q_1 \neq q_2$, it would imply that
\[ |\sqrt{q_1} - \sqrt{q_2}| = |p_1^{b_{1}} - p_2^{b_{2}}| > 1, \]
contradicting the fact that $\qq$ is Hasse.

Case (b): Assume without loss of generality that $p_1$ ramifies in $K$.
From Case~(a) we know that only one of the exponents $a_i$ can be even.

(i) Let us first assume that $a_1$ is odd.
Since $p_1$ ramifies in $\O_K$, this implies that $p_1$ should divide the fundamental discriminant $D$ of $K$, and therefore the power of $p_1$ dividing $\D$ must be odd. We recall that in the case of odd Hasse pairs, the fundamental discriminant is not divisible by $4$ and is therefore squarefree.
From the following relation 
\begin{equation*}
    t_1^2 - 4q_1 = \D = f^2D,
\end{equation*}
it follows that $p_1 | t_1$.
Therefore, there exist $k_{1} \in \mathbb{Z}$ and $b_{1} \in \Z_{>0}$ such that $t_1 = k_{1}p_1^{b_{1}}$, with $\gcd(k_{1},p_1)=1$.
By considering the parity of the exponents of $p_i$ in the above equation, we have that $2b_1>a_1$. 
However, since $\D < 0$ and $p_1$ is odd, this case can only happen when $p_1 = 3$ and $|t_1| = 3^{\frac{a_1+1}{2}}$. 
This is case \eqref{W:odd2} of Theorem~\ref{thm:Waterhouse}, which implies that $\E_1$ is supersingular.
Proposition~\ref{prop:ss} now implies that $\text{gcd}(\D, p_2) = 1$, from which we conclude that $\D = t_2^{2} - 4q_2$ is a non-zero quadratic residue modulo $p_2$, and so is $D$ since $D \equiv (\frac{t_2}{f})^{2} \bmod p_2$. Therefore $p_2$ must split in $K$ by \cite[Proposition 5.16]{Cox}.

(ii) Assume now that $a_1$ is even, so $a_2$ is odd. Then, by Case~(a) $p_2$ cannot be inert. From Case (b-i) above, we see that $p_2$ cannot ramify either, because this would imply that $p_1$ splits, contradicting our assumption that $p_1$ ramifies. Hence, $p_2$ must split in $K$.
\endproof

\begin{remark} \label{rmk:split} 
We note that if $\E_i$ is ordinary, then $p_i$ splits because $\D \equiv t_i^{2} \bmod p_i$. The opposite, however, might not hold, since $p_i$ may split and divide the conductor $f$ of $\D$.
In particular, Proposition~\ref{PrimeSplits} alone does not guarantee the existence of an ordinary elliptic curve in $\E_1 \cup \E_2$.
\end{remark}

In the case of odd Hasse pairs $(p_1,p_2)$ with $p_1,p_2$ pure primes, one can conclude that both of them split completely in $K$ unless $(p_1,p_2) = ((3,7))$ (Example \ref{ex:3-7bis}).
However, this does not have to be true for prime powers with exponent $> 1$ (Example \ref{ex:notsplit}).
Thus, the splitting of one $p_i$ given by Proposition~\ref{PrimeSplits} is the best one can guarantee for Hasse pairs in general.

\begin{ex} \label{ex:3-7bis}
    Let us consider the Hasse pair $(3,7)$ as in Example \ref{ex:ord-ss}, whose discriminant is $\D = -3$.
    By Remarks \ref{rmk:pp} and \ref{rmk:split}, this is the unique odd Hasse pair consisting of pure primes that are not both guaranteed to split completely.
    In fact, $( 3 ) = ( \sqrt{-3} )^2 \subset \O_K$ is ramified, while $( 7 )$ splits because $\E_2$ is ordinary (Remark \ref{rmk:split}).
\end{ex}

\begin{ex} \label{ex:notsplit}
    Since $\sqrt{307} = 17.521\dots$, then $(17^2,307)$ is an odd Hasse pair, whose discriminant is $\D = -3 \cdot 17^2$.
    As $\E_2$ is ordinary by Lemma \ref{lem:TypeSqr}, then $( 307 )$ splits completely in $K$ (Remark \ref{rmk:split}).
    On the other hand, $( 17 )$ is prime (inert) in $K$.
\end{ex}

For the remainder of this section, we study under which conditions $\E_1 \cup \E_2$ may be empty.

\begin{defn} Let $\O$ be an order with fraction field $K$.
We say that an ideal $\mathfrak{a} \subset \O$ is \emph{proper} whenever $\O = \{ b \in K \ | \ b\mathfrak{a} = \mathfrak{a} \}$.
We recall that principal ideals are always proper \cite[p. 135]{Cox}.
A proper $\O$-ideal is called \emph{primitive} if it is not of the form $d\mathfrak{a}$ for some proper ideal $\mathfrak{a}$ and $d \in \Z_{>0}$.
We will call an element $\lambda \in \O$ primitive if $(\lambda)$ is a primitive ideal.  \end{defn}

\begin{prop} \label{NotEmpty}
Let $\qq$ be an odd Hasse pair with corresponding traces $t_i=k_{i}p_i^{b_{i}}$, where $b_{i} \in \N$ and $\gcd(k_i,p_i)=1$. Then one of the sets $\E_i$ is ordinary, except possibly when $b_{i} \geq 1$, $a_i > 2b_i$ and $p_i^{2b_{i}} || \D$ for both $i \in \{1, 2\}$. 
\end{prop}
\proof By Proposition~\ref{prop:ss} it is sufficient to show that if the above conditions are not satisfied, then $\E_1 \cup \E_2$ is non-empty. We show the result in steps.

\textbf{Step 1:} If either one of the $\pi_{i}$'s is primitive, then $\E_i$ is non-empty:

(i): If $\gcd(t_i,\D) = 1$, then $\E_i$ is ordinary by Lemma~\ref{lem:TypeSqr}.

(ii): If $\gcd(t_i,\D) > 1$, since $\pi_{i}$ is primitive then in particular $p_i$ cannot divide the conductor. This implies that $p_i$ must divide $D$ and hence ramifies.
From Case(b-i) of Proposition~\ref{PrimeSplits}, this can only happen when $p_i = 3$ and $|t_i| = 3^{\frac{a_i+1}{2}}$, which implies that $\E_i$ is supersingular. 

\textbf{Step 2:} If $2 b_i \geq a_i$, then $\E_1 \cup \E_2$ is non-empty:

We rewrite the norm equation
\[ t_i^{2} - 4q_i = p_i^{a_{i}} (k_{i}^{2}p_i^{2b_{i}-a_{i}} - 4) = \D < 0 \ \Rightarrow \ k_{i}^{2}p_i^{2b_{i}-a_{i}} < 4. \]
Then we must have $k_{i} = \pm 1$, and either 

(i) $2b_{i} > a_{i}$, in which case we should have $p_i = 3$ and $\E_i$ is supersingular by case \eqref{W:odd2} of Theorem~\ref{thm:Waterhouse}, or

(ii) $2b_{i} = a_{i}$, which implies that $2b_{j} > a_j$ and this is case (i) right above.

\textbf{Step 3:} Given Steps 1 and 2, $\E_1 \cup \E_2$ may be empty only if neither $\pi_{i}$ is primitive and $2 b_{i} < a_{i}$ for both $i \in \{1,2\}$. Given the norm equation, this again implies that $p_i^{2b_i} || \D$ for both $i \in \{1 , 2\}$.
\endproof

\begin{corollary}\label{cor:TwoPower}
If $a_i \leq 2$ for some $i \in \{1,2\}$, then at least one of the sets $\E_i$ is ordinary.
\end{corollary}

\proof
If for some $i \in \{1,2\}$ we have $a_i \leq 2$, then we cannot have $a_i > 2 b_i$ and $b_i \geq 1$, hence we are not in the exceptional case of Proposition~\ref{NotEmpty}.
\endproof

A necessary condition for an odd Hasse pair to satisfy the exceptional conditions of Proposition~\ref{NotEmpty} is having positive integers $a_i > 2 b_i$ and odd integer primes $p_i$ such that 
\[ \begin{cases}
    p_1^{a_1} \equiv 1 \bmod p_2^{b_2}, \\
    p_2^{a_2} \equiv 1 \bmod p_1^{b_1}, \\
    \left| \sqrt{p_1^{a_1}}-\sqrt{p_2^{a_2}} \right| \leq 1. \\
\end{cases} \]

Even for $b_i = 1$, we have found no evidence of such values for moderately small prime powers ($p_i \leq 10^5$ and $a_i \leq 10^3$).
Moreover, large prime powers are expected to be asymptotically “very far apart” \cite{Pillai1,Pillai2}, therefore finding large solutions to the above inequality seems hard.


\section{The isogeny graphs of ordinary Hasse pairs} \label{sec:ord-ord}

\subsection{The graph vertices: Sets of $j$-invariants}\label{Subs:Aux}

In this section, we consider the case where both sets $\E_i$ are ordinary and we will refer to the pair $\qq$ as an \emph{ordinary Hasse pair}.
In such cases, $\gcd(\D, q_1q_2) = 1$ and $\D$ is odd, therefore the fundamental discriminant $D$ is squarefree.

By Lemma~\ref{lem:ord-ord2} there are only two non-odd ordinary Hasse pairs, up to permutation, and it may be directly verified that their isogeny graphs (as will be defined in Section \ref{subsec:IsoGraphs}) are isomorphic (see Examples \ref{ex:2-4} and \ref{ex:4-8}).
Therefore, we can restrict to odd ordinary Hasse pairs and use our definitions and results from Section~\ref{split}. 

We recall from Remark~\ref{rmk:SameOrders} that the order $\mathbb{Z}[\pi_{1}] = \O_{\D} = \mathbb{Z}[\pi_{2}]$ is of discriminant $\D = f^{2}D$. 
For any positive divisor $g$ of $f$, there corresponds an order $\O_{g^{2}D}$, with discriminant $g^2D | \D$ and ideal class group $\Cl(\O_{g^{2}D})$ \cite[Lemma 7.2]{Cox}.
The cardinality $h_{g^{2}D} = |\Cl(\O_{g^{2}D})|$ of the ideal class group is known as the class number of $\O_{g^{2}D}$.
We denote by $H_{g^{2}D}(x)$ the Hilbert class polynomial of $\O_{g^{2}D}$ and by $L_{g^{2}D}$ the splitting field of $H_{g^{2}D}(x)$, also known as the ring class field of $\O_{g^{2}D}$. 
Given any set of representatives $\{\mathfrak{a}_{s}\}_{1 \leq s \leq h_{g^{2}D}}$ for $\Cl(\O_{g^{2}D})$, the Hilbert class polynomial $H_{g^{2}D}(x)$ is given by 
\[
H_{g^{2}D}(x) = \prod_{1 \leq s \leq h_{g^{2}D}} \big(x - j(\mathfrak{a}_{s})\big) \in \Z[x].
\]
We also have a group isomorphism $\Cl(\O_{g^{2}D}) \simeq \Gal(L_{g^{2}D}/K)$,
which can be explicitly described as follows \cite[Theorem 5, Chapt. 10, \S 3]{Lang}:
\begin{equation*} 
\mathfrak{b} \mapsto \sigma_{\mathfrak{b}}, \quad \textnormal{where} \quad \sigma_{\mathfrak{b}}j(\mathfrak{a}) = j(\mathfrak{b}^{-1}\mathfrak{a}).
\end{equation*}
In particular, the ideal class group $\Cl(\O_{g^{2}D})$ acts on the set $\{j(\mathfrak{a}_{s})\}_{1 \leq s \leq h_{g^{2}D}}$ freely and transitively, and therefore  $L_{g^{2}D}=K\big(j(\mathfrak{a})\big)$ for any $\mathfrak{a} \in \Cl(\O_{g^{2}D})$.

\begin{lemma}\label{PrimitiveFrobs}
    Let $\qq$ be an ordinary Hasse pair. The ideals $(\pi_{i}) \subset \O_K$ 
    are proper, primitive, and $(\pi_{i}) = \mathfrak{p}_{i}^{a_{i}}$, where the $\mathfrak{p}_{i}$'s are prime ideals such that $\mathfrak{p}_{i}\overline{\mathfrak{p}}_{i} = (p_i)$ and $\mathfrak{p}_{i} \neq \overline{\mathfrak{p}}_{i}$.
\end{lemma}
\proof
Since we are in the ordinary case we must have that $\gcd(t_i,\D) = 1$, which implies that $p_i \nmid t_i$ and $p_i \nmid f$ for both $i\in\{1, 2\}$.
As a consequence \cite[Lemma, \S4, p. 537]{Waterhouse} both primes $p_i$ must split in $K$ as $p_i = \mathfrak{p}_{i}\overline{\mathfrak{p}}_{i}$, $\mathfrak{p}_{i} \neq \overline{\mathfrak{p}}_{i}$. 
Let us consider the norm equation
\[ (q_i) = (\pi_{i})(\overline{\pi}_{i}) = \mathfrak{p}_{i}^{a_{i}} \overline{\mathfrak{p}}_{i}^{a_{i}}. \]
Since the norms of $\mathfrak{p}_{i}^{a_{i}}$ and $(\pi_{i})$ are prime to the conductor, these ideals are therefore primitive and they belong to the group of ideals with unique factorization in the order $\O_{\D} \subseteq \O_{K}$.
We conclude that $(\pi_{i}) = \mathfrak{p}_{i}^{m_{i}} \overline{\mathfrak{p}}_{i}^{n_{i}}$, for some $m_{i}, n_{i} \in \N$ such that $m_{i} + n_{i} = a_{i}$.
Assume by way of contradiction that both $m_{i}, n_{i}$ are not zero. Then $p_i | \pi_{i}$ and therefore $p_i | t_i = \pi_{i}+\overline{\pi}_{i}$, contradiction.
Therefore we have that either $m_{i} = 0$ or $n_{i} = 0$, so up to conjugation we have 
$ (\pi_{i}) = \mathfrak{p}_{i}^{a_{i}} \ \text{in} \  \O_{\D} \subseteq \O_{K}$. 
\endproof

\begin{lemma} \label{lem:HilSplits}
    Let $\qq$ be an ordinary Hasse pair.
    For every $g\in\Z_{>0}$ with $g|f$, $H_{g^{2}D}$ splits completely with distinct roots in both $\mathbb{F}_{q_i}[x]$.
    In particular, $H_{\D}$ splits completely with distinct roots in both $\F_{q_i}[x]$.
\end{lemma}
\proof
   Let us fix primes $\mathfrak{B}_{i}$ of $\O_{L}$ above $\mathfrak{p}_{i}$. 
   Since our primes $p_i$ neither ramify nor divide the conductor, then by \cite[Proposition I.8.3]{Neuk}, the Hilbert class polynomial $H_{\D}$ splits completely with distinct roots in both finite fields $\F_{p_i^{\iota_{i}}}$, where $\iota_{i}$ is called the inertia degree and it is the degree of the residue field extension 
   \begin{equation} \label{eq:inertiadegree}
    \iota_{i} = [\O_{L_{\D}} / \mathfrak{B}_{i} : \O_{\D} / \mathfrak{p}_{i}].
   \end{equation} 
   Since the inertia degree is the smallest power to which the base ideal, $\mathfrak{p}_{i}$ in this case, becomes principal in $\Cl(\O_{\D})$, and since $\mathfrak{p}_{i}^{a_{i}} = (\pi_{i})$ by Lemma~\ref{PrimitiveFrobs}, then $\iota_{i} | a_{i}$.
   From the fact that $h_{g^{2}D} | h_{\D}$, we have that the primes $\mathfrak{p}_{i}\O_{g^{2}D}$ must also become principal when raised to $a_{i}$, for both $i\in\{1, 2\}$. Hence, by the same \cite[Proposition I.8.3]{Neuk}, $H_{g^{2}D}$ must also split completely with distinct roots over both $\mathbb{F}_{q_i}[x]$.
\endproof

For each $g | f$, we denote by $J_{g^{2}D}$ the set of complex roots of the corresponding Hilbert Class Polynomial $H_{g^{2}D}$. We note that two different $H_{g^{2}D}$ cannot have common roots, and 
we denote by $J$ the (disjoint) union of all these sets.
We similarly denote by $J_{{(g^{2}D)}_{i}}$ and $J_i$ the sets of roots of the same polynomials in $\F_{q_i}$.

\begin{remark}\label{rmk:card1}
Since $H_{g^2D}$ splits completely by Lemma \ref{lem:HilSplits}, the cardinality of each $J_{g^{2}D}$ is $h_{g^2D}$, i.e.
\[ |J_{g^{2}D}| = |J_{{(g^{2}D)}_{1}}| = |J_{{(g^{2}D)}_{2}}|. \]
Therefore, we also have
\[ |J| = |J_{1}| = |J_{2}|. \]
We also remark that, given Lemma \ref{lem:HilSplits}, the sets $J_{{(g^{2}D)}_{i}}$ are already defined in $\F_{p_i^{\iota_i}} \hookrightarrow \F_{q_i}$.
\end{remark}

Since we are in the ordinary case, by Remark~\ref{rmk:uniquej} each set $\E_i$ can be identified with the set of $j$-invariants of their elliptic curves. We note that this cannot be done in the supersingular case (Example \ref{ex:4-7}).
Remark~\ref{rmk:card1} now implies that we can identify the sets $\E_i$  with $J_i$.

\begin{lemma} \label{lem:goodproj}
    Let $j^{(0)} \in L$ be a root of $H_{\D}$ and $r$ be a fixed square root of $\D$. Let also $D_0 = \textnormal{Disc}(H_{\D}) \in \Z$.
    For every $j \in J$, there exists $h_j(x,y) \in \Z[x,y]$ such that
    \[ j = \frac{h_j(j^{(0)},r)}{2\,D_0}. \]
    Moreover, if $p$ is an odd prime such that $H_{\D}$ does not have repeated roots over $\bar{\F}_p$, then $\gcd(2\,D_0,p)=1$.
\end{lemma}
\proof For every order $\O_{\D} \se \O \se \O_K$ and every $j \in J$, we have $j \in L_{g^{2}D} \se L = K(j^{(0)})$ \cite[Theorem 10.6]{Lang} and $H_{\O}(j) = 0$.
Since $H_{\O}$ is a monic integral polynomial, then $j \in \O_L$.
By classical properties of simple separable extensions \cite[Lemma I.2.9]{Neuk}, we have
\[\O_L \se D_0^{-1} \O_K[ j^{(0)} ] \se (2D_0)^{-1} \Z[ j^{(0)}, r ].\]
Thus, every $j \in \O_L$ can be written as $\frac{h_j(j^{(0)},r)}{2 D_0}$ for some $h_j(x) \in \Z[x,y]$.
Moreover, since $H_{\D}$ is the minimal polynomial of $j^{(0)} \in L$ over $\Q$ \cite[Theorem II.4.3]{SilvermanAdv}, we have
\[ D_0 = \pm \textnormal{Disc}(H_{\D}) = \pm \textnormal{Res}_x \big(H_{\D}(x),H_{\D}'(x)\big) \in \Z.\]
Then $H_{\D}$ has no repeated roots in $\bar{\F}_p$ precisely when $D_0 = \pm \textnormal{Res}_x \big(H_{\D}(x),H_{\D}'(x)\big) \in \bar{\F}_p^*$, i.e. when $D_0$ is coprime to $p$.
\endproof

\begin{remark} \label{rmk:Bij}
We fix once and for all a square root $r$ of $\D \in \mathbb{C}$, and an element $j^{(0)} \in J_{\D}$, which generates $L_{\D}$ over $K$.
Now, given Remark \ref{rmk:card1} and Lemma \ref{lem:goodproj}, any fixed choice of a square root $r_{i}$ of $\D \in \mathbb{F}_{p_i^{\iota_{i}}}$ and $j^{(0)}_{i} \in J_{(\D)_{i}}$, induces an explicit bijection 
\[\Pi_{i} : J \to J_{i} \subseteq \F_{p_i^{\iota_{i}}}, \quad \frac{h_j(j^{(0)},r)}{2D_0} \mapsto \frac{h_j(j^{(0)}_{i},r_i)}{2D_0}.\]
By embedding $\F_{p_i^{\iota_i}}$ into $\mathbb{F}_{q_i}$, this extends to a bijection $\Pi_{i} : J \to J_{i} \subseteq \F_{q_i}$, which we denote in the same way. 
We notice that the above choices of $r_i$ and $j^{(0)}_{i}$ are associated with a fixed choice of primes $\P_i \subset \O_L$ and $\mathfrak{p}_{i} \subset \O_K$ such that $\mathfrak{P}_{i} | \mathfrak{p}_{i} | p_i$, for both $i\in\{1, 2\}$. 
\end{remark}

\subsection{The Isogeny Graphs} \label{subsec:IsoGraphs}

\begin{defn}
    Let $\B \subset \N$ be a finite set of integer primes, and $\E$ be a set of elliptic curves defined over a finite field $\k$, up to $\k$-isomorphism.
    We define the \emph{$\B$-isogeny graph} of $\E$, denoted by $\GB(\E)$, as the directed colored multigraph, whose vertices are the elements in $\E$, and the edges are cyclic $\k$-isogenies between them.
    The color of any edge corresponds to the degree of the isogeny, which is some prime in $B$.
\end{defn}

\begin{remark}
    When drawing isogeny graphs we use the convention that an undirected edge corresponds to one edge in each direction. 
    We recall that $j=1728$ can appear in an ordinary $\GB(\E)$ only in even characteristic (so it equals $j = 0$) by Remark \ref{rmk:no1728}.
    Since every isogeny has a dual isogeny \cite[Section III.6]{Silverman}, and $\Aut(E) = \{\pm \textnormal{id}\}$ for every elliptic curve $E$ with $j(E) \neq 0, 1728$ \cite[Section III.10]{Silverman}, our graphs may be thought of as undirected, with the only possible exception of isogenies from $j=0$.
\end{remark}

\begin{lemma}\label{lem:FrobVer}
Let $\qq$ be an ordinary Hasse pair.
For both $i \in \{1,2\}$, the graphs $G_{\{p_i\}}(\E_1)$ and $G_{\{p_i\}}(\E_2)$ are isomorphic and consist of $h_{\D}/\iota_i$ cycles of length $\iota_i$ (as defined in \eqref{eq:inertiadegree}).
\end{lemma}
\proof
Since the relation is symmetric, we can assume without loss of generality that $i=2$.
We can also restrict to vertices contained in $J_{(\D)_1}$ and $J_{(\D)_2}$, because all the $p_2$-isogenies are horizontal since $\qq$ is ordinary. 
We compute the eigenvalues of the $\pi_{i}$'s by examining the splitting of their minimal polynomials $f_{i}(x) \bmod p_2$. Given Remark \ref{rmk:FrobsPrim}, a simple computation shows that
\[ f_{1}(x) = f_{2}(-x+1). \]
We compute the eigenvalues $\{\mu_{i}, \nu_{i}\}$ of each $\pi_{i}$ as the roots of
\begin{align*}
    f_{2}(x) &\equiv x(x-t_2) \bmod p_2, \\
    f_{1}(x) = f_{2}(-x+1) &\equiv (x - 1)\big(x-(1-t_2)\big) \bmod p_2.
\end{align*}
For ordinary elliptic curves in characteristic $p_2$, there are only two isogenies of degree $p_2$: the Frobenius isogeny, which is inseparable, and its dual, called the \emph{Verschiebung}, which is separable.
The Frobenius isogeny creates cycles of length equal to the degree of the minimal extension field where the elliptic curve is defined, which in this case is $\iota_{2}$ as in \eqref{eq:inertiadegree}. 
Therefore, the number of cycles in this case is $h_{\D}/\iota_{2}$.

We now look at the eigenvalues $\lambda_{1} = 1$ and $\mu_{1} = 1- t_2$ of $f_{1}$.
We see that $\mu_{1} \not\equiv 0 \bmod p_2$ since otherwise, from the equation $q_1 = q_2+1-t_2$, we would have that $p_2 | q_1$, which is not possible for odd Hasse pairs.
Furthermore, $\lambda_{1} \not\equiv \mu_{1} \bmod p_2$, as $t_2 \equiv 0 \bmod p_2$ would lead to supersingular curves. 
Since we have two distinct eigenvalues, this implies that $p_2$ is an Elkies prime and we have $h_{\D}/\iota_2$ cycles of degree-$p_2$ isogenies in this case as well (\cite[Theorem 7]{SutherlandVolcanoes}).
Thus, $G_{\{p_2\}}(\E_1)$ and $G_{\{p_2\}}(\E_2)$ have the same structure.
\endproof

\begin{thm} \label{thm:graphsIso}
Let $\qq$ be an ordinary Hasse pair and $B \subset \N$ be a finite set of primes.
Then $\GB(\E_1)$ is isomorphic to $\GB(\E_2)$.
\end{thm}
\proof 
Let $\GB(\D)$ be the directed colored multigraph (in characteristic zero) whose vertices are the elements of $J$ and whose edges $\phi: j_1 \to j_2$ represent isogenies of prime color $\deg \phi = \ell$ between curves with $j$-invariants $j_1$ and $j_2$, for any $\ell \in B$. 

By adapting \cite[Proposition II.4.4]{SilvermanAdv} to the context of Remark~\ref{rmk:Bij}, for both $i\in\{1, 2\}$ we have an injection 
\[ \Hom( j_1, j_2 ) \to \Hom\big( \Pi_i(j_1), \Pi_i(j_2) \big), \]
which preserves the degree of the isogenies.
Therefore, there is an isomorphic copy of $\GB(\D)$ over both the $\F_{p_i^{\iota_{i}}}$'s.
Since the isogenies, as well as the $j$-invariants involved, are defined over $\F_{p_i^{\iota_{i}}}$, they are also defined over $\F_{q_i}$, therefore the maps $\Pi_i$ are isomorphic graph embeddings of $\GB(\D)$ into the $\GB(\E_i)$'s.

Since all the considered curves are ordinary, the structure of the isogeny graphs $\GB(\D)$, $\GB(\E_1)$, and $\GB(\E_2)$ only depends on $\D$.
Therefore, the number and type of isogenies (horizontal or vertical) from curves with $j$-invariant different from $0$ are the same in both characteristics (by \cite[Theorem 7]{SutherlandVolcanoes} and \cite[Proposition 23]{Kohel}) for isogenies of degree $l \not\in \{p_1,p_2\}$.
By Lemma \ref{lem:FrobVer}, the same holds for isogenies of degree $p_1$ and $p_2$. 


Finally, either none or both sets $\E_i$ contain a curve $E^0_i$ with $j(E^0_i) = 0$, as this occurs if and only if the squarefree part of $\D$ is $D = -3$.
If such curves exist, then we have
\[ \End_{\F_{p_1}}(E^0_1) \simeq \End_{\F_{p_2}}(E^0_2) \simeq \Z\left[\frac{1+\sqrt{-3}}{2}\right]. \]
Since a $j$-invariant zero curve over $\F_{p_i^{a_i}}$ is ordinary if and only if $p_i \equiv 1 \bmod 3$, then $\sqrt{-3}$ is defined over both the $\F_{p_i}$'s (and consequently over the $\F_{q_i}$'s).
Thus, in both the $\GB(\E_i)$'s, if we have any descending $l$-isogenies $0 \to j \neq 0$, then there are precisely three distinct $l$-isogenies from $0$ to $j$.
Moreover, in this case, both the isogeny graphs have exactly $(\frac{\D}{l})+1$ isogenies of degree $l$ from $0$ to itself.
\endproof

\section{Examples of isogeny graphs} \label{examples}

\begin{ex} \label{ex:4-7}
    The discriminant of the Hasse pair $\qq = (4,7)$ is $\D = -12$.
    Given $\F_4 = \F_2(\gamma) \simeq \F_2[x]/(x^2+x+1)$, one can directly verify that $\E_1$ contains only the following two elliptic curves
    \begin{equation} \label{eq:0A0B}
        0_A \ : \ y^2 + \gamma y = x^3 + 1, \quad \quad
        0_B \ : \ y^2 + \gamma^2 y = x^3 + \gamma^2,
    \end{equation} 
    which are both supersingular and their $j$-invariant is $0 \in \F_4$.
    On the other hand, $\E_2$ is ordinary and contains the curves of $j$-invariant $0$ and $2 \in \F_7$.
    The $2$-Frobenius map $\pi_2$ connects the curves $0_A$ and $0_B$ in $\textnormal{G}_{\{2\}}(\E_1)$, while there are three descending and one ascending $2$-isogenies in $\textnormal{G}_{\{2\}}(\E_2)$.
    Thus, the isogeny graphs are not isomorphic as shown in Figure \ref{fig:NonIso}.
\begin{figure}[h]
    \centering
    \begin{tikzpicture}[scale=0.9]
        \node (1) at (-1,-2) {$0_A$};
        \node (2) at (1,-2) {$0_B$};
        \draw[blue!50!black] (1) -- (2) node[midway,above] {$\pi_2$};
    \end{tikzpicture} \qquad \qquad \qquad \qquad
    \begin{tikzpicture}[scale=0.9]
        \node (1) at (0,0) {$0$};
        \node (2) at (0,-2) {$2$};
        \draw[->] (1) edge[bend right=28, blue!50!black] (2);
        \draw (1) edge[blue!50!black](2);
        \draw[->] (1) edge[bend left=28, blue!50!black] (2);
    \end{tikzpicture}
    \caption{$\{\textcolor{blue}{2}\}$-isogeny graphs for the Hasse pair $(4,7)$. The curves $0_A$ and $0_B$ are defined as in \eqref{eq:0A0B}.}
    \label{fig:NonIso}
\end{figure}
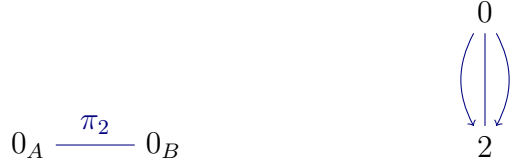
\end{ex}

\begin{ex} \label{ex:2-4}
    The pair $\qq = (2,4)$ is an ordinary Hasse pair with associated discriminant $\D = -7$.
    Given $\F_4 = \F_2(\gamma) \simeq \F_2[x]/(x^2+x+1)$, one can directly verify that $\E_1$ only contains the curve $E_1$ defined by $y^2 + xy = x^3 + 1$ over $\F_2$, while $\E_2$ only contains the curve $E_2$ defined by $y^2 + xy = x^3 + \gamma^2 x^2 + 1$ over $\F_4$, both of $j$-invariant $1$.
    If $\pi_2$ is the $2$-Frobenius map, one can also check that
    \[ \End_{\F_2}(E_1) = \Z[\pi_2] = \End_{\F_4}(E_2), \]
    so the isogeny graphs are isomorphic and given in Figure \ref{fig:Ex2}.
\begin{figure}[h]
    \centering
    \begin{tikzpicture}[scale=0.9]
        \node (1) at (0,4) {$1$};
        \node (pi) at (0,5) {\textcolor{green!50!black}{$\pi_2$}};
        \draw[->,green!50!black] (1)++(0.2,0)arc(-60:90:0.4);
        \draw[-,green!50!black] (1)++(0.2,0)arc(-60:240:0.4);
    \end{tikzpicture} \qquad \qquad \qquad \qquad
    \begin{tikzpicture}[scale=0.9]
        \node (1) at (0,4) {$1$};
        \node (pi) at (0,5) {\textcolor{green!50!black}{$\pi_2$}};
        \draw[->,green!50!black] (1)++(0.2,0)arc(-60:90:0.4);
        \draw[-,green!50!black] (1)++(0.2,0)arc(-60:240:0.4);
    \end{tikzpicture}
    \caption{$\B$-isogeny graphs $\GB(\E_1)$ (on the left) and of $\GB(\E_2)$ (on the right) for $\qq \in \{((2,4)),((4,8))\}$ for any $2 \in \B \se \N$.}
    \label{fig:Ex2}
\end{figure}
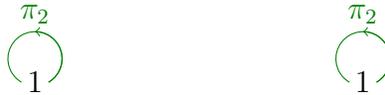
\end{ex}

\begin{ex} \label{ex:4-8}
    The pair $\qq = (4,8)$ is the second of the two ordinary non-odd Hasse pair of Lemma \ref{lem:ord-ord2}, the other one being Example \ref{ex:2-4}.
    The associated discriminant is again $\D = -7$.
    One can directly verify that both sets $\E_i$ contain only the curve defined by the model $y^2 + xy = x^3 + 1$, of $j$-invariant $1$.
    Notice that in $\E_1$ we have the (unique) twist of $E_2$ from Example \ref{ex:2-4}.
    Therefore the endomorphism ring for both curves is $\Z[\pi_2]$, and so the isogeny graphs are isomorphic and are shown again in Figure \ref{fig:Ex2}.
\end{ex}

\begin{ex}
    Since $\sqrt{587} = 24.228\dots$, then $\qq = (5^4,587)$ is an odd Hasse pair, which is ordinary since its associated discriminant $\D = -979 = -11 \cdot 89$ is coprime with both the $q_i$'s (Lemma~\ref{lem:TypeSqr}).
    Thus, the isogeny graphs of $\E_1$ and $\E_2$ are the same by Theorem~\ref{thm:graphsIso}.
    This correspondence is portrayed for isogenies of prime degree up to $11$ by Figure \ref{fig:IsoGraphs}.

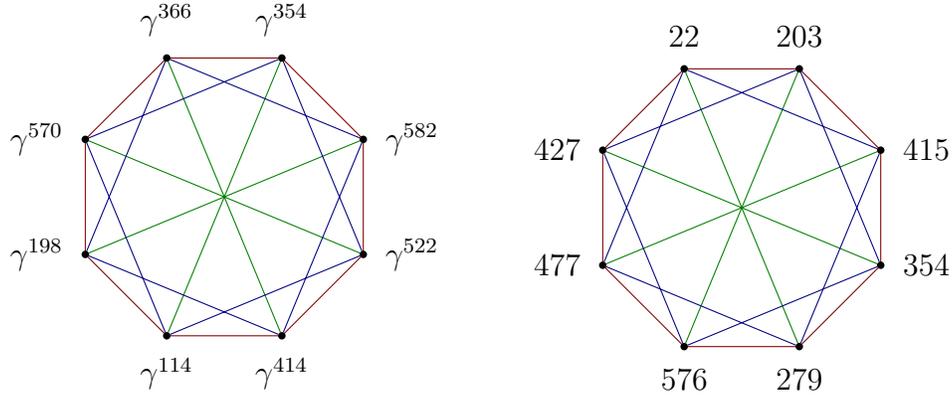
\begin{figure}[h]
    \centering
    \tikzset{
        summit/.style={inner sep=1pt,outer sep=0pt,circle,fill=black,text=white},
        }
    \begin{tikzpicture}
        \newdimen\R
        \R=2cm
        \draw[red!50!black] (-22.5:\R) \foreach \x [count=\i] in {0,45,...,315} {  -- (\x+22.5:\R) node[summit] (\i) {} };
        
        \begin{scope}[on background layer]
            \draw[green!50!black]
                (1) -- (5)
                (2) -- (6)
                (3) -- (7)
                (4) -- (8);
            \draw[blue!50!black]
                (1) -- (3) -- (5) -- (7) -- (1)
                (2) -- (4) -- (6) -- (8) -- (2);
        \end{scope}
        \node [label=right:{$\gamma^{582}$}] at (1) {};
        \node [label=above:{$\gamma^{354}$}] at (2) {};
        \node [label=above:{$\gamma^{366}$}] at (3) {};
        \node [label=left:{$\gamma^{570}$}] at (4) {};
        \node [label=left:{$\gamma^{198}$}] at (5) {};
        \node [label=below:{$\gamma^{114}$}] at (6) {};
        \node [label=below:{$\gamma^{414}$}] at (7) {};
        \node [label=right:{$\gamma^{522}$}] at (8) {};
    \end{tikzpicture} \qquad
    \begin{tikzpicture}
        \newdimen\R
        \R=2cm
        \draw[red!50!black] (-22.5:\R) \foreach \x [count=\i] in {0,45,...,315} {  -- (\x+22.5:\R) node[summit] (\i) {} };
        
        \begin{scope}[on background layer]
            \draw[green!50!black]
                (1) -- (5)
                (2) -- (6)
                (3) -- (7)
                (4) -- (8);
            \draw[blue!50!black]
                (1) -- (3) -- (5) -- (7) -- (1)
                (2) -- (4) -- (6) -- (8) -- (2);
        \end{scope}
        \node [label=right:{$415$}] at (1) {};
        \node [label=above:{$203$}] at (2) {};
        \node [label=above:{$22$}] at (3) {};
        \node [label=left:{$427$}] at (4) {};
        \node [label=left:{$477$}] at (5) {};
        \node [label=below:{$576$}] at (6) {};
        \node [label=below:{$279$}] at (7) {};
        \node [label=right:{$354$}] at (8) {};
    \end{tikzpicture}
    \caption{$\{$2,\textcolor{white!50!black}{3},\textcolor{blue!50!black}{5},\textcolor{red!50!black}{7},\textcolor{green!50!black}{11}$\}$-isogeny graphs for the Hasse pair $(5^4,587)$.
    The vertices are $j$-invariants over $\F_{5^4} = \F_{5}[\gamma] \simeq \F_5[x]/(x^4 + 4x^2 + 4x + 2)$ (left) and over $\F_{587}$ (right).} 
    \label{fig:IsoGraphs}
\end{figure}
\end{ex}

\begin{ex}
    Since $\sqrt{1069} - \sqrt{1021} = 0.742\dots$, then $\qq = (1021,1069)$ is an odd Hasse pair.
    Its associated discriminant is $\D = -1875 = -3 \cdot 5^4$, so both sets $\E_i$ are ordinary by Lemma~\ref{lem:TypeSqr}, and their isogeny graphs are the same by Theorem~\ref{thm:graphsIso}.
    This correspondence is portrayed for isogenies of certain degrees by Figures \ref{fig:IsoGraphsj0A} and \ref{fig:IsoGraphsj0B}.

\begin{figure}[h]
    \centering
        \begin{tikzpicture}[scale=0.9]
        \node (1) at (0,4) {$0$};
        \node (2) at (-2.5,2) {$89$};
        \node (3) at (2.5,2) {$277$};
        \node (4) at (-4.0,0) {$33$};
        \node (5) at (-3.1,0) {$109$};
        \node (6) at (-2.2,0) {$462$};
        \node (7) at (-1.3,0) {$730$};
        \node (8) at (-0.4,0) {$894$};
        \node (9) at (0.5,0) {$931$};
        \node (10) at (1.3,0) {$782$};
        \node (11) at (2.2,0) {$912$};
        \node (12) at (3.1,0) {$143$};
        \node (13) at (4.0,0) {$439$};

        \draw[->,green!50!black] (1)++(0.2,0)arc(-60:90:0.4);
        \draw[-,green!50!black] (1)++(0.2,0)arc(-60:240:0.4);
        
        \draw[->] (1) edge[bend right=20, blue!50!black] (2);
        \draw[-] (1) edge[blue!50!black] (2);
        \draw[->] (1) edge[bend left=20, blue!50!black] (2);
        
        \draw[->] (1) edge[bend right=20, blue!50!black] (3);
        \draw[-] (1) edge[blue!50!black] (3);
        \draw[->] (1) edge[bend left=20, blue!50!black] (3);
        

        \draw[-] (2) edge[blue!50!black] (4);
        \draw[-] (2) edge[blue!50!black] (5);
        \draw[-] (2) edge[blue!50!black] (6);
        \draw[-] (2) edge[blue!50!black] (7);
        \draw[-] (2) edge[blue!50!black] (8);

        \draw[-] (3) edge[blue!50!black] (9);
        \draw[-] (3) edge[blue!50!black] (10);
        \draw[-] (3) edge[blue!50!black] (11);
        \draw[-] (3) edge[blue!50!black] (12);
        \draw[-] (3) edge[blue!50!black] (13);

        \draw[-] (2) edge[green!50!black] (3);

        \draw[-] (4.south) edge[bend right, green!50!black] (13.south);
        \draw[-] (5.south) edge[bend right, green!50!black] (12.south);
        \draw[-] (6.south) edge[bend right, green!50!black] (11.south);
        \draw[-] (7.south) edge[bend right, green!50!black] (10.south);
        \draw[-] (8.south) edge[bend right, green!50!black] (9.south);
    \end{tikzpicture}
    \begin{tikzpicture}[scale=0.9]
        \node (1) at (0,4) {$0$};
        \node (2) at (-2.5,2) {$643$};
        \node (3) at (2.5,2) {$855$};
        \node (4) at (-4.0,0) {$1025$};
        \node (5) at (-3.1,0) {$883$};
        \node (6) at (-2.2,0) {$364$};
        \node (7) at (-1.3,0) {$266$};
        \node (8) at (-0.4,0) {$79$};
        \node (9) at (0.5,0) {$212$};
        \node (10) at (1.3,0) {$537$};
        \node (11) at (2.2,0) {$195$};
        \node (12) at (3.1,0) {$707$};
        \node (13) at (4.0,0) {$352$};

        \draw[->,green!50!black] (1)++(0.2,0)arc(-60:90:0.4);
        \draw[-,green!50!black] (1)++(0.2,0)arc(-60:240:0.4);

        \draw[->] (1) edge[bend right=20, blue!50!black] (2);
        \draw[-] (1) edge[blue!50!black] (2);
        \draw[->] (1) edge[bend left=20, blue!50!black] (2);
        
        \draw[->] (1) edge[bend right=20, blue!50!black] (3);
        \draw[-] (1) edge[blue!50!black] (3);
        \draw[->] (1) edge[bend left=20, blue!50!black] (3);
        

        \draw[-] (2) edge[blue!50!black] (4);
        \draw[-] (2) edge[blue!50!black] (5);
        \draw[-] (2) edge[blue!50!black] (6);
        \draw[-] (2) edge[blue!50!black] (7);
        \draw[-] (2) edge[blue!50!black] (8);

        \draw[-] (3) edge[blue!50!black] (9);
        \draw[-] (3) edge[blue!50!black] (10);
        \draw[-] (3) edge[blue!50!black] (11);
        \draw[-] (3) edge[blue!50!black] (12);
        \draw[-] (3) edge[blue!50!black] (13);

        \draw[-] (2) edge[green!50!black] (3);

        \draw[-] (4.south) edge[bend right, green!50!black] (13.south);
        \draw[-] (5.south) edge[bend right, green!50!black] (12.south);
        \draw[-] (6.south) edge[bend right, green!50!black] (11.south);
        \draw[-] (7.south) edge[bend right, green!50!black] (10.south);
        \draw[-] (8.south) edge[bend right, green!50!black] (9.south);
    \end{tikzpicture}
    \caption{\{\textcolor{green!50!black}{3},\textcolor{blue!50!black}{5}\}-isogeny graph for the ordinary Hasse pair $(1021,1069)$. The vertices are $j$-invariants over $\F_{1021}$ (left) and $\F_{1069}$ (right).}
    \label{fig:IsoGraphsj0A}
\end{figure}
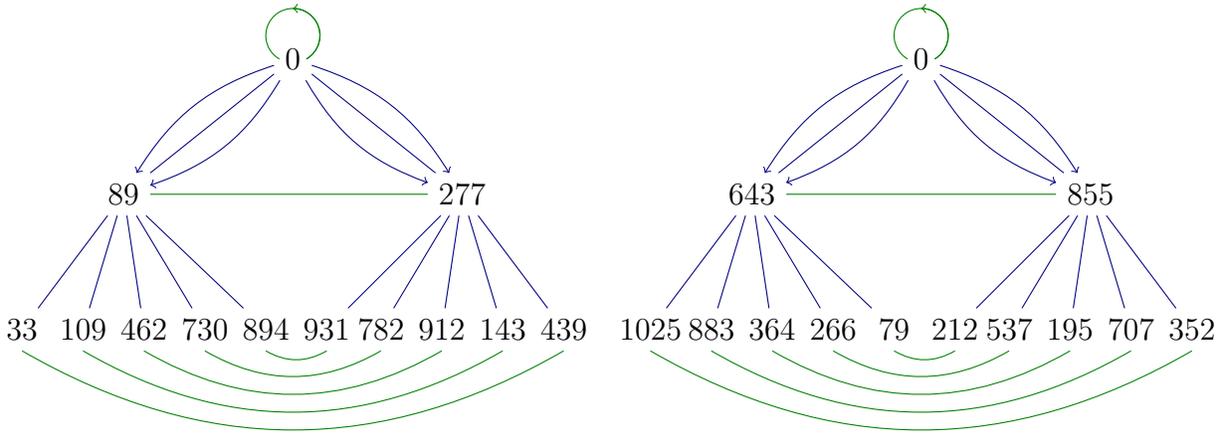

\begin{figure}[h]
    \centering
\begin{tikzpicture}[scale=0.9]
        \node (1) at (0,2.8) {$0$};
        \node (2) at (-1.5,1.3) {$89$};
        \node (3) at (1.5,1.3) {$277$};
        \node (4) at (-3.0,0) {$143$};
        \node (5) at (-1,0) {$462$};
        \node (6) at (1,0) {$782$};
        \node (7) at (3,0) {$894$};
        \node (8) at (4,-1) {$439$};
        \node (9) at (3,-2) {$109$};
        \node (10) at (1,-2) {$912$};
        \node (11) at (-1,-2) {$730$};
        \node (12) at (-3,-2) {$931$};
        \node (13) at (-4,-1) {$33$};

        \draw[red!50!black] ([xshift=0.15cm]1) arc(-65:250:0.4);

        \draw[-] (2) edge[bend right=15, red!50!black] (3);
        \draw[-] (3) edge[bend right=15, red!50!black] (2);
        
        \draw[-] (4) edge[red!50!black] (5);
        \draw[-] (5) edge[red!50!black] (6);
        \draw[-] (6) edge[red!50!black] (7);
        \draw[-] (7) edge[bend left, red!50!black] (8);
        \draw[-] (8) edge[bend left, red!50!black] (9);
        \draw[-] (9) edge[red!50!black] (10);
        \draw[-] (10) edge[red!50!black] (11);
        \draw[-] (11) edge[red!50!black] (12);
        \draw[-] (12) edge[bend left, red!50!black] (13);
        \draw[-] (13) edge[bend left, red!50!black] (4);

        \draw[lime!50!black] ([xshift=0.15cm]1) arc(65:-250:0.4);

        \draw[lime!50!black] ([xshift=-0.2cm,yshift=0.27cm]2) arc(45:315:0.4);

        \draw[lime!50!black] ([xshift=0.2cm,yshift=0.27cm]3) arc(135:-135:0.4);

        \draw[-] (5) edge[bend right, lime!50!black] (7);
        \draw[-] (7) edge[lime!50!black] (9);
        \draw[-] (9) edge[bend right, lime!50!black] (11);
        \draw[-] (11) edge[lime!50!black] (13);
        \draw[-] (13) edge[lime!50!black] (5);

        \draw[-] (4) edge[bend right, lime!50!black] (6);
        \draw[-] (6) edge[lime!50!black] (8);
        \draw[-] (8) edge[lime!50!black] (10);
        \draw[-] (10) edge[bend right, lime!50!black] (12);
        \draw[-] (12) edge[lime!50!black] (4);
    \end{tikzpicture}
    \begin{tikzpicture}[scale=0.9]
        \node (1) at (0,2.8) {$0$};
        \node (2) at (-1.5,1.3) {$643$};
        \node (3) at (1.5,1.3) {$855$};
        \node (4) at (-3.0,0) {$537$};
        \node (5) at (-1,0) {$364$};
        \node (6) at (1,0) {$707$};
        \node (7) at (3,0) {$79$};
        \node (8) at (4,-1) {$352$};
        \node (9) at (3,-2) {$266$};
        \node (10) at (1,-2) {$195$};
        \node (11) at (-1,-2) {$883$};
        \node (12) at (-3,-2) {$212$};
        \node (13) at (-4,-1) {$1025$};

        \draw[red!50!black] ([xshift=0.15cm]1) arc(-65:250:0.4);
        
        \draw[-] (2) edge[bend right=15, red!50!black] (3);
        \draw[-] (3) edge[bend right=15, red!50!black] (2);
        
        \draw[-] (4) edge[red!50!black] (5);
        \draw[-] (5) edge[red!50!black] (6);
        \draw[-] (6) edge[red!50!black] (7);
        \draw[-] (7) edge[bend left, red!50!black] (8);
        \draw[-] (8) edge[bend left, red!50!black] (9);
        \draw[-] (9) edge[red!50!black] (10);
        \draw[-] (10) edge[red!50!black] (11);
        \draw[-] (11) edge[red!50!black] (12);
        \draw[-] (12) edge[bend left, red!50!black] (13);
        \draw[-] (13) edge[bend left, red!50!black] (4);

        \draw[lime!50!black] ([xshift=0.15cm]1) arc(65:-250:0.4);

        \draw[lime!50!black] ([xshift=-0.2cm,yshift=0.27cm]2) arc(45:315:0.4);

        \draw[lime!50!black] ([xshift=0.2cm,yshift=0.27cm]3) arc(135:-135:0.4);

        \draw[-] (5) edge[bend right, lime!50!black] (7);
        \draw[-] (7) edge[lime!50!black] (9);
        \draw[-] (9) edge[bend right, lime!50!black] (11);
        \draw[-] (11) edge[lime!50!black] (13);
        \draw[-] (13) edge[lime!50!black] (5);

        \draw[-] (4) edge[bend right, lime!50!black] (6);
        \draw[-] (6) edge[lime!50!black] (8);
        \draw[-] (8) edge[lime!50!black] (10);
        \draw[-] (10) edge[bend right, lime!50!black] (12);
        \draw[-] (12) edge[lime!50!black] (4);
    \end{tikzpicture}
    \caption{\{\textcolor{red!50!black}{7},\textcolor{lime!50!black}{19}\}-isogeny graph for the ordinary Hasse pair $(1021,1069)$. The vertices are $j$-invariants over $\F_{1021}$ (left) and $\F_{1069}$ (right).}
    \label{fig:IsoGraphsj0B}
\end{figure}
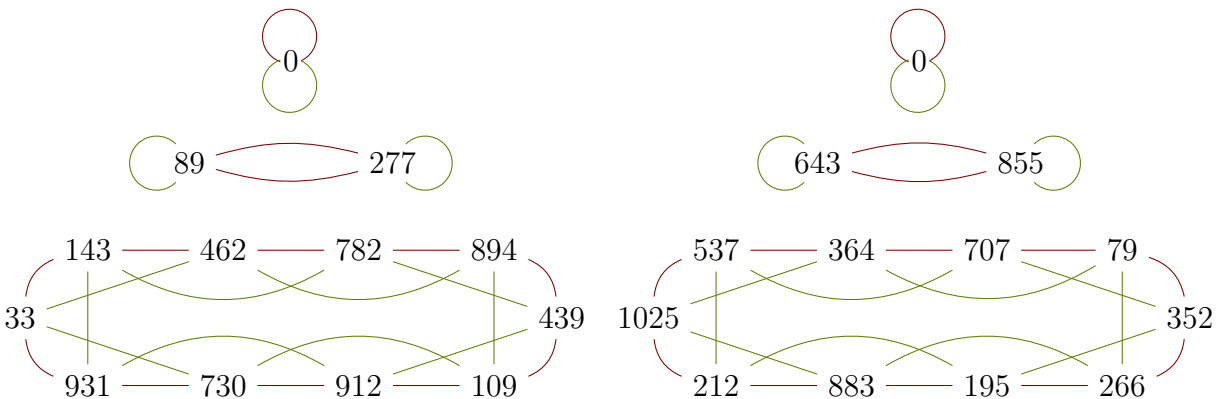
\end{ex}

\begin{ex}
    Since $\sqrt{22501} = 150.003\dots$, then $\qq = (151^2,22501)$ is an odd Hasse pair.
    Since $\D = -507 = -3^2 \cdot 67$, all the curves are ordinary and the isogeny graphs are isomorphic by Theorem~\ref{thm:graphsIso}.
    Those graphs are shown for isogenies of certain degrees by Figure \ref{fig:IsoGraphsAnom}.
    We observe that all the $j$-invariants of $\E_1$ are already defined over $\F_{151}$, since $\iota_1 = 1$ (as defined in \eqref{eq:inertiadegree}). 
    In fact, one can verify that the resulting isogeny graph equals the isogeny graph of (isogeny classes of) anomalous curves over $\F_{151}$. 

\begin{figure}[h]
    \centering
        \begin{tikzpicture}[scale=0.9]
        \node (1) at (0,3.5) {$125$};
        \node (2) at (-2.5,2) {$40$};
        \node (3) at (2.5,2) {$75$};
        \node (4) at (2.5,0) {$99$};
        \node (5) at (-2.5,0) {$96$};

        \draw[-,blue!50!black] (1)++(0.34,0.1)arc(-45:225:0.45);
        \draw[-,red!50!black] (1)++(0.44,0.1)arc(-45:225:0.6);

        \draw[-] (1) edge[green!50!black] (2);
        \draw[-] (1) edge[green!50!black] (3);
        \draw[-] (1) edge[green!50!black] (4);
        \draw[-] (1) edge[green!50!black] (5);

        \draw[-] (2) edge[blue!50!black] (3);
        \draw[-] (3) edge[blue!50!black] (4);
        \draw[-] (4) edge[blue!50!black] (5);
        \draw[-] (5) edge[blue!50!black] (2);

        \draw[-] (2) edge[bend right=5, red!50!black] (4);
        \draw[-] (2) edge[bend left=5, red!50!black] (4);
        \draw[-] (3) edge[bend right=5, red!50!black] (5);
        \draw[-] (3) edge[bend left=5, red!50!black] (5);
        
    \end{tikzpicture} \qquad
    \begin{tikzpicture}[scale=0.9]
        \node (1) at (0,3.5) {$12341$};
        \node (2) at (-2.5,2) {$337$};
        \node (3) at (2.5,2) {$620$};
        \node (4) at (2.5,0) {$4861$};
        \node (5) at (-2.5,0) {$6330$};

        \draw[-,blue!50!black] (1)++(0.31,0.2)arc(-45:225:0.45);
        \draw[-,red!50!black] (1)++(0.425,0.2)arc(-45:225:0.6);

        \draw[-] (1) edge[green!50!black] (2);
        \draw[-] (1) edge[green!50!black] (3);
        \draw[-] (1) edge[green!50!black] (4);
        \draw[-] (1) edge[green!50!black] (5);

        \draw[-] (2) edge[blue!50!black] (3);
        \draw[-] (3) edge[blue!50!black] (4);
        \draw[-] (4) edge[blue!50!black] (5);
        \draw[-] (5) edge[blue!50!black] (2);

        \draw[-] (2) edge[bend right=5, red!50!black] (4);
        \draw[-] (2) edge[bend left=5, red!50!black] (4);
        \draw[-] (3) edge[bend right=5, red!50!black] (5);
        \draw[-] (3) edge[bend left=5, red!50!black] (5);
    \end{tikzpicture}
    \caption{\{\textcolor{green!50!black}{3},\textcolor{blue!50!black}{17},\textcolor{red!50!black}{19}\}-isogeny graph for the ordinary Hasse pair $(151^2,22501)$. The vertices are $j$-invariants over $\F_{151} \hookrightarrow \F_{151^2}$ (left) and $\F_{22501}$ (right).}
    \label{fig:IsoGraphsAnom}
\end{figure}
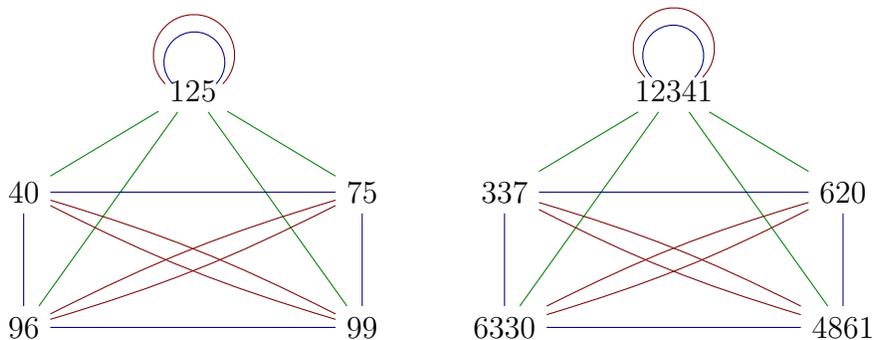
\end{ex}


\section{Conclusions} \label{conclusions}
In the present paper, for a given Hasse pair $\qq$ we investigated the properties of the sets $\E_1$ and $\E_2$. 

When the pair is odd, the results we obtained are summarized in Table \ref{table:odd-odd}.
In this case, we proved that $\E_1 \cup \E_2$ always contains an ordinary curve when it is non-empty, but we could not prove nor disprove whether such a union could be empty.
This remains an open problem for us (Problem \ref{prob:empty-empty}), and we believe that its difficulty lies in the fact that the prime powers involved should be very close, as well as meeting further modular conditions. 

\begin{table}[h]
    \ct
    \begin{tabular}{ |c|c|C{3.8cm}|C{3.8cm}|C{3.8cm}| } 
 \hline
   \multicolumn{2}{|c|}{ \textbf{$(q_1,q_2)$} } & \multicolumn{3}{c|}{ $\E_1$ } \\ 
   \cline{3-5}
   \multicolumn{2}{|c|}{ \textbf{odd Hasse pair} } & ordinary & supersingular & empty \\ 
 \hline
 
 & ordinary & Same isogeny graphs & Occurs with $j=0$ & May occur \\
  & & Theorem~\ref{thm:graphsIso} & \ct Proposition~\ref{prop:j=0}
  & E.g. $(7^2,43)$ \\ 
 \cline{2-5}
 $\E_2$ & supersingular & \cellcolor{lightgray} & \multicolumn{2}{c|}{ Never occurs } \\ 
   & & \cellcolor{lightgray} Symmetric & \multicolumn{2}{c|}{ Proposition~\ref{prop:ss} } \\ 
 \cline{2-2} \cline{4-5}
 & empty & \multicolumn{2}{c|}{\cellcolor{lightgray} see $(q_2,q_1)$ \hspace{4cm}\phantom{.}} & Open, \\ 
  & & \multicolumn{2}{c|}{ \cellcolor{lightgray} } & Problem \ref{prob:empty-empty} \\ 
  
 \hline
\end{tabular}
    \caption{Possible configuration arising from odd Hasse pairs.}
    \label{table:odd-odd}
\end{table}

When the considered Hasse pair involves even entries, different cases may occur, as portrayed in Table \ref{table:even-odd}.
In this setting, it is possible to find only supersingular curves in $\E_1 \cup \E_2$.
These cases arise precisely when we have Fermat and Mersenne primes, therefore counting such cases is equivalent to the well-known open problem of deciding the finiteness of these families of primes (Remark \ref{rmk:FermatMersenne}).
Furthermore, the set $\E_i$ defined in even characteristic may be empty only if the other one (defined in odd characteristic) is also empty, but we have found only one instance of such a case.
Again, the empty cases look special, and establishing a complete classification of them seems challenging.

\begin{table}[h]
    \ct
\begin{tabular}{ |c|c|C{3.8cm}|C{3.8cm}|C{3.8cm}| } 
 \hline
   \multicolumn{2}{|c|}{ \textbf{$(2^{a_1},q_2)$} } & \multicolumn{3}{c|}{ $\E_1$ } \\
   \cline{3-5}
   \multicolumn{2}{|c|}{ \textbf{Hasse pair} } & ordinary & supersingular & empty \\ 
 \hline
 & ordinary & $q_2$ even, classified & May occur  & May occur \\
  & & Lemma~\ref{lem:ord-ord2} & Example \ref{ex:4-7} 
  & E.g. $(2^4,11)$ \\ 
 \cline{2-5}
 $\E_2$ & supersingular &  & Classified & May occur \\ 
   & & Never occurs & Theorem~\ref{thm:even-ss} & E.g. $(2^3,7)$ \\ 
 \cline{2-2} \cline{4-5}
 & empty & Lemma~\ref{lem:ord-ord2} & Never occurs & May occur \\ 
   & &  & Theorem~\ref{thm:even-ss} & E.g. $(2^8,3^5)$ \\ 
 \hline
\end{tabular}
\caption{Possible configuration arising from non-odd Hasse pairs.}
    \label{table:even-odd}
\end{table}

Furthermore, in the last section we observed that, regardless of the characteristic, when $\E_1$ and $\E_2$ are both ordinary, we have a perfect correspondence between their isogeny graphs.
This relation is effective only on paper at the moment, as we lack an efficient way of passing from one isogeny graph to the other.
Performing such correspondence efficiently, i.e. without passing through the construction of huge Hilbert class polynomials in characteristic zero, is regarded as an ambitious task, which could lead to interesting cryptographic applications.

Finally, similarly to how aliquot cycles extend the idea of amicable primes \cite{SilvStan}, one may aim at extending the ideas of the current paper by considering the tuples $(p_1^{a_1}, \dots, p_r^{a_r})$ of prime powers and elliptic curves $E_1, \dots, E_r$, such that for every $1 \leq i \leq r$, the curve $E_i$ is defined over $\F_{q_i}$ and
\[ \forall\, 1 \leq i \leq r-1 \ : \ |E_i| = q_{i+1}, \quad |E_r| = q_1. \]
When ordinary, all these curves should have the same isogeny graphs.

\section*{Acknowledgements}
This work was partially supported by the European Union’s H2020 Programme under grant agreement number ERC-669891. D.T. was also supported by the Research Foundation - Flanders (FWO mandate: 12ZZC23N).
E.A. would like to thank the CISPA Helmholtz Center for Information Security, and the Max Planck Institute for Mathematics in Bonn, for their hospitality and support during the writing of this paper. The authors are grateful to Pieter Moree for helpful discussions, insightful comments, and careful proof-reading of the paper.


\appendix

\section{Hasse pairs and Andrica's conjecture \\ by Pieter Moree and Efthymios Sofos}\label{Ap:PietSof}

Recall that a pair of prime
powers $(q_1,q_2)$
is called a \emph{Hasse pair} if
\begin{equation}
\label{Hinequality}
\mid\sqrt{q_1}-\sqrt{q_2}\mid
\le 1.
\end{equation}
If both $q_1$ and $q_2$ are odd, then the
pair is called an \emph{odd Hasse pair}.
Using Lemma \ref{lem:consecutiveprimepowers} it is easy to see that if $\sqrt{q_1}-\sqrt{q_2}=1$, then
necessarily $q_1=81$ and $q_2=64$, thus
replacing $\le 1$ by $<1$ is almost equivalent.

It is natural to ask if there are many Hasse pairs.
In this direction we contribute the following
result.
\begin{thm}
\label{main}
Fix any $\epsilon>0$. There exists a constant $c(\epsilon)$ such that
for every prime  $p<x$ and real number $x\ge 2$, with at most $c(\epsilon)\,x^{3/5+\epsilon}$
exceptions, there are at least $\frac12
\sqrt{p}\,(\log p)^{-1}$  Hasse pairs
$(p,\ell)$ with $\ell$ a prime.
\end{thm}
We will apply some ideas from
analytic number theory to this question. Since the number of primes up
to $x$ grows asymptotically as $x/\log x$ and the number of
prime powers $p^m$ with $m\ge 2$ only as $\sqrt{x}$, which is
$\le c(\epsilon)\,x^{3/5+\epsilon}$ for every $x$ sufficiently large, we see
that Theorem \ref{main} implies the following corollary.
\begin{corollary}
Theorem \ref{main} also holds true if instead of primes we consider prime powers.
\end{corollary}
Thus typically a large prime (power) occurs as member of a \emph{large number}
of Hasse pairs.
However, we cannot exclude the existence of a prime power
that \emph{never} occurs as a member of a Hasse pair.
\begin{Con}
\label{hassecon}
Every prime power is member of some Hasse pair.\\
Every odd prime power is member of some odd Hasse pair.
\end{Con}
In Section \ref{sec:Andrica} we show that this conjecture is intimately
connected with Andrica's conjecture.
Research on that conjecture then shows that Conjecture \ref{hassecon} is
certainly true for all prime powers $<10^{19}$.
Let $p_1=2, p_2=3,\ldots$ be the sequence of
consecutive primes.
\begin{prop}
If Andrica's conjecture that $(p_j,p_{j+1})$ is always a Hasse pair is
true, then so is Conjecture \ref{hassecon}.
\end{prop}
\begin{proof}
If $q$ is a prime power, we let $p_j$ be the largest prime $\le q$.
Then $(q,p_{j+1})$ is a Hasse pair by Andrica's conjecture.
If $q$ is odd, then $(q,p_{j+1})$ is an odd Hasse pair. \end{proof}

\subsection{Proof of Theorem \ref{main}}
Theorem \ref{main} is in essence a corollary of the following
deep result due to Heath-Brown \cite{primegap}.
\begin{thm}[Heath-Brown \cite{primegap}]
\label{t.HB2}
For any fixed $\epsilon>0$ the measure of the set of $y\in [0,x]$ such that
\[ \max_{0\leq h \leq \sqrt y}\left|\pi(y+h)-\pi(y)-\int_y^{y+h} \frac{\mathrm d t}{\log t}\right|\geq \frac{\sqrt y}{(\log y)(\log \log y)} \]
is $O_\epsilon(x^{3/5+\epsilon})$.
\end{thm}


\noindent \textit{Proof of Theorem \ref{main}:} Let $B_x$ be the set of primes $p\in [2,x]$
such that the interval $(p,p+\sqrt{p})$ contains $\le \sqrt{p}/(2\log p)$ primes.
Our plan is to show that for any $p\in B_x\cap [2,x-1]$ the interval
  $(p,p+1)$ contributes towards the measure in Theorem~\ref{t.HB2}, from which we can directly
deduce that  the number of elements in $B_x$ is $c(\epsilon)\,x^{3/5+\epsilon}$.

For any $p\in B_x$ and $y\in (p,p+1)$
we have that the number of primes in the interval
$(y,y+\sqrt{p})$ is at most
2
plus the number of primes in the interval $(p,p+\sqrt{p})$, which by
assumption is $\le
2+\sqrt{p}/(2\log p)$.
Now for all sufficiently large $p$ by trivial calculus this is
$\le \sqrt{p}/(\log 2p) - \sqrt{p+1
}/(\log {(p+1
) \log\log(p+1
) })$
and by the monotonicity of the log function this is
\[ \le \int_y^{y+\sqrt{p}}\frac{\mathrm dt}{\log t}-\frac{\sqrt{p+1}}{\log(p+1) \log \log(p+1) },\]
but since
$y<p+1$,
this is
\[ \le \int_y^{y+\sqrt{p}}\frac{\mathrm dt}{\log t}-\frac{\sqrt y}{\log y\log \log y}.\]
Now taking $h=\sqrt{p}$ we note that $h\leq \sqrt{y}$
and hence we have shown that if $p\in B_x$, then every $y$ in the
interval $(p,p+1)$ violates the inequality in Theorem~\ref{t.HB2}.\hfill{$\square$}


\subsection{Andrica's conjecture}
\label{sec:Andrica}
Let $d_n:=p_{n+1}-p_n$ be the $n$th prime gap.
Over the centuries many conjectures involving
$d_n$ have been made. Here we are interested in a conjecture
made in 1986 by Dorin Andrica \cite{Andrica}.
\begin{Con}[Andrica \cite{Andrica}]
The inequality
\begin{equation}
\label{Andr1}
\sqrt {p_{n+1}}-\sqrt {p_{n}}<1
\end{equation}
holds for all $n$.
\end{Con}
A lot of numerical work on large gaps has
been done.
This can be used to infer that
Andrica's conjecture holds for $p_n\le 2\cdot 2^{63}\approx 1.8\cdot
10^{19}$; see Visser \cite{Visser}.

An equivalent formulation of Andrica's conjecture is that
\begin{equation}
\label{Andr2}
d_{n}<2\sqrt {p_{n}}+1
\end{equation}
holds for all $n$. It seems to be far out of reach, as under RH the best
result is due to Cram\'er \cite{Cramer} who showed in
already 1920 that $d_n=O(\sqrt{p_n}\log p_n)$. More explicitly,
Carneiro et al.\,\cite{CMS} showed under RH that $d_n\le
\frac{22}{25}\sqrt{p_n}\log p_n$
for every $p_n>3$. Unconditionally Baker et
al.\,\cite{BHP} showed that $d_n<p_n^{0.525}$ for all $n$ sufficiently
large.

Heath-Brown used Theorem \ref{t.HB2}  to obtain the
following result on large gaps between consecutive primes,
which is a culmination of an enormous amount of research done over
many years.
\begin{thm}[Heath-Brown \cite{primegap}]
\label{t.HB}
We have
\[\sum_{\substack{p_n\le x\\ p_{n+1}-p_n\ge\sqrt{p_n}}}(p_{n+1}-p_n)\ll_{\epsilon} x^{3/5+\epsilon}.\]
\end{thm}
For a table of earlier (larger) exponents than $3/5$ achieved and their authors, see Moree \cite{NAW}.

The authors (with Kosyak and Zhang) \cite{KMSZ} came across the quantity
on the left hand side in Theorem \ref{t.HB} as essentially being the number of failures up to $x$ for a
method of construction they have for producing cyclotomic polynomials
$\Phi_n(x)$ having maximum coefficient $h$, with $h\le x$.
If a slightly stronger variant of Andrica's conjecture holds true, then they showed
that \emph{every} natural number occurs as maximum coefficient of some cyclotomic polynomial.
For a more
leisurely account, see Moree \cite{NAW}.

\subsection{Andrica's conjecture for denser sequences}
In our context a slightly weaker version of Andrica's conjecture comes
up, namely:
\begin{Con}[Andrica for prime powers]
Let $q_1,q_2,\ldots$ be the consecutive prime powers.
The inequality
\begin{equation}
\label{AndrPP}
\sqrt {q_{n+1}}-\sqrt {q_{n}}<1
\end{equation}
holds for all $n$.
\end{Con}
Since the number of primes up to $x$ grows asymptotically as $x/\log x$
and the number of
prime powers $p^m$ with $m\ge 2$ as $\sqrt{x}$, this modified version
is, unfortunately, not that much weaker.
One could hope that if one replaces the sequence of prime numbers by a
substantially `fatter' sequence, one might be
able to prove the analogue of Andrica's conjecture. Finally, we present
an example showing that this hope is not unjustified.

A positive integer $n$ is
said to be \emph{practical} if every positive integer $\le n$
is a subset-sum of the positive divisors of
$n$.
Weingartner \cite{wein}
showed that there is a constant
$c = 1.33607\ldots$ such that the number of
practical numbers $\le x$ grows as $cx/\log x$.
The analogue of Andrica's conjecture for the practical numbers is known
to be true, see
Wu \cite{Wu}.

\medskip\noindent {\footnotesize Max-Planck-Institut f\"ur Mathematik,\\
Vivatsgasse 7, D-53111 Bonn, Germany.\\
E-mail: {\tt moree@mpim-bonn.mpg.de}}\\

\medskip\noindent{\footnotesize School of Mathematics and Statistics,\\
University of Glasgow,\\
University Place, Glasgow, G12 8SQ, United Kingdom.\\
E-mail: {\tt efthymios.sofos@glasgow.ac.uk}}\\


\bibliographystyle{amsplain}

\newpage


\end{document}